\documentclass[11pt,reqno,a4paper]{amsart} 
\usepackage{graphicx}
\usepackage{amssymb,amsmath}
\usepackage{amsthm}
\usepackage{color,graphicx}
\usepackage{hyperref}
\usepackage{color}
\usepackage{mathabx}
\usepackage{subfigure}
\usepackage[pagewise]{lineno}

\usepackage{subfigure} 

\usepackage{verbatim} 

\usepackage{relsize} 

\newcommand{\be}{\begin{equation}}

\newcommand{\ee}{\end{equation}}

\newcommand{\sech}{{\rm \,sech}}
\newcommand{\R}{{\mathbb R}}

\newcommand{\Lum}{{\mathcal{L}_{-}}}

\newcommand{\ve}{{\varepsilon}}

\newcommand{\Hh}{\mathbb{H}}

\newcommand{\Ker}{{\rm \,Ker}}

\numberwithin{equation}{section}
\numberwithin{figure}{section}

\newtheorem{theorem}{Theorem}[section]
\newtheorem{proposition}[theorem]{Proposition}
\newtheorem{remark}[theorem]{Remark}
\newtheorem{lemma}[theorem]{Lemma}
\newtheorem{corollary}[theorem]{Corollary}
\newtheorem{definition}[theorem]{Definition}

\begin{document}
\vglue-1cm \hskip1cm
\title[orbital stability of solitary waves]{On the orbital stability of solitary waves for the fourth order nonlinear Schr\"odinger equation}

\begin{center}



\maketitle
%
%
\author{\textbf{Handan Borluk}\footnote{\texttt{handan.borluk@ozyegin.edu.tr}} \\  
	{\small Department of Mathematical Engineering, Ozyegin University},
	{\small  Cekmekoy, Istanbul,  T\" urkiye} \\
	\textbf{Gulcin M. Muslu}\footnote{\texttt{gulcin@itu.edu.tr}} \\ 
	{\small  
		Istanbul Technical University, Department of Mathematics, Maslak 34469, Istanbul,  T\" urkiye}\\   
	\textbf{F\'abio Natali}\footnote{\texttt{fmanatali@uem.br}} \\ 
	{\small Department of Mathematics, State University of Maring\'a, Maring\'a, PR, Brazil}\\
}



\vspace{3mm}

\end{center}

\begin{abstract}
In this paper, we present new results regarding the orbital stability of solitary standing waves for the general fourth-order Schr\"odinger equation with mixed dispersion. The existence of solitary waves can be determined both as minimizers of a constrained complex functional and by using a numerical approach. In addition, for specific values of the frequency associated with the standing wave, one obtains explicit solutions with a hyperbolic secant profile. Despite these explicit solutions being minimizers of the constrained functional, they cannot be seen as a smooth curve of solitary waves, and this fact prevents their determination of stability using classical approaches in the current literature. To overcome this difficulty, we employ a numerical approach to construct a smooth curve of solitary waves. The existence of a smooth curve is useful for showing the existence of a threshold power $\alpha_0\approx 4.8$ of the nonlinear term  such that if $\alpha\in (0,\alpha_0),$ the explicit solitary wave is stable, and if $\alpha>\alpha_0$, the wave is unstable. An important feature of our work, caused by the presence of the mixed dispersion term, concerns the fact that the threshold value $\alpha_0 \approx 4.8$ is not the same as that established for proving the existence of global solutions in the energy space, as is well known for the classical nonlinear Schr\"odinger equation.

\end{abstract}

\section{Introduction} 
Consider the fourth order nonlinear Schr\"odinger equation (4NLS)
\begin{equation}\label{NLS-equation}
	i u_t + u_{xx}-u_{xxxx} + |u|^\alpha u = 0,
\end{equation}
where $u:\mathbb{R}\times \mathbb{R}\rightarrow\mathbb{C}$ and $\alpha>0$. Since $u$ is a complex function, we sometimes write $u=u_1+iu_2\equiv (u_1,u_2)$ in order to simplify the notation. The physical motivation concerning particular cases of the equation $(\ref{NLS-equation})$ was introduced in \cite{kar} and \cite{kar1}. In fact, when $\alpha=2$, it is related to the propagation of intense laser beams in a bulk medium with Kerr nonlinearity when small fourth-order dispersion is taken into account. Moreover, equation $(\ref{NLS-equation})$ has been considered in connection with nonlinear fiber optics and the theory of optical solitons in gyrotropic media. An interesting fact arises when we examine equation $(\ref{NLS-equation})$ from a mathematical standpoint, as it does not enjoy scaling invariance due to the presence of the mixed dispersion term. This property is very important for the standard nonlinear Schr\"odinger equation (NLS),
\begin{equation}\label{2NLS-equation}
	i u_t + u_{xx}+ |u|^\alpha u = 0.
\end{equation}
\indent Let $\lambda>0$ be fixed. The scaling invariance $u\mapsto \lambda^{\frac{2}{\alpha}}u(\lambda x,\lambda^2t)$ applied to the equation $(\ref{2NLS-equation})$ is useful, for instance, to know the connection between the existence of global solutions in $\mathbb{H}^1=H^1(\mathbb{R})\times H^1(\mathbb{R})$ (see \cite[Chapter 6]{linares-ponce}) and the orbital stability of solitary waves with explicit hyperbolic secant profile. In fact, if $0<\alpha<4$, the solution $u$ of the Cauchy problem associated to $(\ref{2NLS-equation})$ is global in time in $\mathbb{H}^1$ by using the Gagliardo-Nirenberg inequality and the associated solitary wave is orbitally stable in the same Sobolev space (see \cite{gss1}, \cite{gss2} and \cite{we}). If $\alpha>4$, the blow-up phenomena in finite time occurs in a convenient set of conditions of the initial data (see \cite[Chapter 6]{linares-ponce}). In addition, when $\alpha>4$ the solitary wave with hyperbolic secant profile is orbitally unstable (see \cite{gss1} and \cite{gss2}). As we will see later on, the model in $(\ref{NLS-equation})$ does not have the same connection, in the sense that we don't have a common threshold value as $\alpha=4$ works for the equation $(\ref{2NLS-equation})$.\\
\indent Equation $(\ref{NLS-equation})$ has the following conserved quantities
\begin{equation}\label{E}
    E(u) = \frac{1}{2} \int_{\mathbb{R}} |u_{xx}|^2+|u_x|^2 - \frac{2}{\alpha + 2} |u|^{\alpha+2} dx,
\end{equation}
and,
\begin{equation}\label{F}
F(u) = \frac{1}{2} \int_{\mathbb{R}} |u|^2 dx.
\end{equation}
In equation \eqref{NLS-equation}, we can consider solitary standing wave solutions of the form $u(x,t) = e^{i\omega t} \varphi(x)$, where $\omega>0$. Substituting this form into \eqref{NLS-equation}, we obtain
\begin{equation}\label{edo}
 \varphi''''-\varphi'' + \omega \varphi - |\varphi|^\alpha \varphi = 0,
\end{equation}
where $\varphi$ is a real function. Concerning equation $(\ref{2NLS-equation})$, we obtain that the unique positive even solution for the equation
\begin{equation}\label{2edo}
 -\varphi'' + \omega \varphi - |\varphi|^\alpha \varphi = 0,
\end{equation}
has the hyperbolic secant profile (as we have already said in the second paragraph) given by

\begin{equation}\label{sechNLS}
\varphi(x)=\left(\omega\frac{\alpha+2}{2}\right)^{\frac{1}{\alpha}}\sech^{\frac{2}{\alpha}}\left(\frac{\alpha\sqrt{\omega}x}{2}\right).
\end{equation}
An important fact from $(\ref{sechNLS})$ is that we obtain an explicit curve $\omega\in(0,+\infty)\mapsto\varphi=\varphi_{\omega}$ of solitary waves depending on $\omega>0$ for each fixed $\alpha>0$. This property is very useful to obtain the stability scenario as determined in \cite{gss1} and \cite{gss2}, mentioned in the second paragraph of this introduction. According to the abstract theories in \cite{gss1} and \cite{gss2}, one of the cornerstone to determine the stability is to analyze the sign of the quantity $d''(\omega)$, where $d(\omega)=E(\varphi,0)+\omega F(\varphi,0)$. In the case of the equation $(\ref{2NLS-equation}),$ we obtain
\begin{equation}\label{d2NLS}
d''(\omega)=\frac{1}{2\omega}\left(\frac{4-\alpha}{2\alpha}\right)\int_{\mathbb{R}}\varphi^2(x)dx.
\end{equation}
Thus, if $0<\alpha<4$, we obtain $d''(\omega)>0$ (orbital stability) and if $\alpha>4$, we obtain \mbox{$d''(\omega)<0$} (orbital instability). When $d''(\omega)=0$ the stability/instability results in \cite{gss1} and \cite{gss2} are inconclusive and an orbital instability can be determined by proving the strong instability (blow-up) as in \cite{cazenave} and \cite{weinstein}. It is important to mention that the arguments in \cite{we} can also be applied only in the case when $d''(\omega)>0$, that is, when the solitary wave is orbitally stable.\\
\indent As far as we know, our contribution has the intention of demonstrating the orbital stability and instability results for specific solutions of $(\ref{edo})$ that have a similar form as in $(\ref{sechNLS})$. Moreover, our results also indicate that the power exponent $\alpha$ does not have a similar connection as we found for equation $(\ref{2NLS-equation})$ and the reason for that is the presence of the mixed dispersion term 
$\partial_x^4-\partial_x^2$. To begin with, we need to show the existence of solitary wave solutions $\varphi$ for the equation \eqref{edo}. First, we find solitary wave solutions by solving the following constrained minimization problem
\begin{equation}\label{minP1}
   \Gamma_{\omega,\tau}= \inf \left\{ \mathcal{B}_\omega(u) := \frac{1}{2} \int_{\mathbb{R}} |u_{xx}|^2+|u_x|^2 + \omega |u|^2 dx;\ u \in \mathbb{H}^2 \text{ and } \int_{\mathbb{R}} |u|^{\alpha+2} dx = \tau \right\},
\end{equation}
where $\tau > 0$ is a fixed number. An application of the concentration compactness method as used in \cite{denis2}, \cite{denis}, and \cite{levandosky}  guarantees the existence of a complex function $\Phi$ such that the minimum of the problem $(\ref{minP1})$ is attained at $\Phi$. Moreover, it follows from $(\ref{edo})$ that $\Phi\in \mathbb{H}^4$ an $\Phi^{(n)}(x)\rightarrow 0$ as $|x|\rightarrow \infty$ for all $n=0,1,2,3$. Assume for a moment that for all $\alpha>0$, function $\Phi$ can be considered as a real even function. We can write $\Phi=\varphi+i0=(\varphi,0)$ or, more generally, as $\Phi=e^{i\theta_0}\varphi$ for some $\theta_0\in\mathbb{R}$. In any case, we obtain $\mathcal{B}_{\omega}(\Phi)=\mathcal{B}_{\omega}(\varphi,0)=\Gamma_{\omega,\tau}$. On the other hand, let us consider $G_{\omega}(u)=E(u)+\omega F(u)$. The fact that $(\varphi,0)$ is a minimizer of $(\ref{minP1})$ enables to consider the linearized operator at the pair $(\varphi,0)$, given by
\begin{equation}\label{operator}
    \mathcal{L}= G_{\omega}''(\varphi,0):=\left( \begin{array}{cc} \mathcal{L}_{-} & 0 \\ 0 & \mathcal{L}_{+} \end{array} \right),
\end{equation}
where $G_{\omega}''(\varphi,0)$ represents the second derivative in the sense of Fr\'echet at the point $(\varphi,0)$ and $\mathcal{L}_{-}$ and $\mathcal{L}_{+}$ are given, respectively by
\begin{equation}\label{L1L2}
\mathcal{L}_{-}=\partial_x^4-\partial_x^2+\omega-(\alpha+1)\varphi^{\alpha}\ \ \ \mbox{and}\ \ \ \ \mathcal{L}_{+}=\partial_x^4-\partial_x^2+\omega-\varphi^{\alpha}.
\end{equation}
\indent According to the stability/instability theories in \cite{gss1} and \cite{gss2} (see also \cite{we}), we need to know the behaviour of the non-positive spectrum of $\mathcal{L}$. In addition, since equation $(\ref{NLS-equation})$ is invariant under translation and rotation, the dimension of $\Ker(\mathcal{L})$ should be at most $2-$dimensional. Thus, if ${\rm n}(\mathcal{L})=1$, ${\rm z}(\mathcal{L})=2$ and $d''(\omega)>0$, we obtain that the solitary wave $\varphi$ that solves equation $(\ref{edo})$ is stable. Now, if ${\rm n}(\mathcal{L})=1$, ${\rm z}(\mathcal{L})=2$ and $d''(\omega)<0$, we obtain the orbital instability. Notations ${\rm n}(\mathcal{L})$ and ${\rm z}(\mathcal{L})$ are the number of negative eigenvalues and the dimension of the kernel of the linear operator $\mathcal{L}$, respectively.\\
\indent On the other hand, explicit solutions for the equation $(\ref{edo})$ are well known but only for a fixed value of $\omega=\omega_0$ depending on $\alpha>0$. They are given by 
\begin{equation}\label{expsolode}
	\phi(x)=a_0\sech^{\frac{4}{\alpha}}(b_0x),
	\end{equation}
	where $a_0$ and $b_0$ are given, respectively, by
	\begin{equation}\label{ab}
		\displaystyle a_0=\left(\frac{3\alpha^3 + 22\alpha^2 + 48\alpha + 32}{2(\alpha^2 + 4\alpha + 8)^2}\right)^{\frac{1}{\alpha}},\ \ \ b_0=\frac{\alpha}{2\sqrt{\alpha^2+4\alpha+8}}.
		\end{equation}
	The frequency $\omega=\omega_0$ of the wave $\phi$ depends only on $\alpha$ and it is expressed as 
	\begin{equation}\label{omega}
		\omega_0=\frac{4(\alpha^2+4\alpha+4)}{\alpha^4+8\alpha^3+32\alpha^2+64\alpha+64}.
		\end{equation}
\indent The explicit solution $\phi$ in $(\ref{expsolode})$ is not a smooth curve depending on $\omega$ in the sense that, for each fixed $\alpha>0$, we have a single solitary wave with a hyperbolic secant profile. We also show in Proposition $\ref{sechmin}$ that $\phi$ solves the minimization problem $(\ref{minP1})$, and thus, we can write $\phi=\varphi$. Additionally, since $\phi = \varphi$ is a minimizer, we obtain that ${\rm n}(\mathcal{L}) = 1$ (see Proposition $\ref{propunique}$). Furthermore, we can conclude that ${\rm n}(\mathcal{L}) = 1$ and verify the crucial requirement for orbital stability, namely ${\rm z}(\mathcal{L}) = 2$, by applying the total positivity theory developed in \cite{al}, \cite{albo}, and \cite{ka} (see Section 3). An application of Theorem $\ref{mainal}$, which is the main result in \cite{al}, yields ${\rm n}(\mathcal{L}_{-})={\rm z}(\mathcal{L}_{-})=1$, and a simple argument of positiveness of the Fourier transform of $\varphi$ also gives ${\rm n}(\mathcal{L}_{+})=0$ and ${\rm z}(\mathcal{L}_{+})=1$. The fact that $\mathcal{L}$ in $(\ref{operator})$ has a diagonal form gives us ${\rm z}(\mathcal{L})={\rm z}(\mathcal{L}_{-})+{\rm z}(\mathcal{L}_{+})=2$. Since we do not know the exact behavior of $\varphi$ in a neighborhood of the point $\omega=\omega_0$ given in $(\ref{omega})$, we cannot calculate $d''(\omega)$ at $\omega=\omega_0$ to determine the sign of this quantity and conclude the orbital stability or instability using \cite{gss1}, \cite{gss2} and \cite{we}.\\
\indent To overcome this difficulty, we need to use a numerical approach to construct a convenient smooth curve of solitary waves depending on $\omega>0$. In fact, to generate the solitary wave solutions of \eqref{NLS-equation} we use  Pethviashvili's iteration method. 
The method is widely used for the generation of traveling wave solutions (\cite{BMN, aduran, LP, moraes, OBM, OBMN, PS}). Therefore, we obtain solutions for $\omega \neq \omega_0$ where the explicit solution is unknown.
It is known that there are non-positive solutions for $\omega > 0.25$,  which is an interesting feature of the solutions for the equation $(\ref{edo})$ since they do not appear for the classical NLS (\cite{PC, sprenger}). We obtain both positive and non-positive solutions numerically. Taking positive solutions for $\omega \leq 0.25$ into account, we calculate the value of $d''(\omega)$ numerically. Our results also point out that for $\alpha \in (0, 4.1)$, we have $d''(\omega)>0$. For $\alpha \in [4.1, 5.5)$, there is a threshold value $\omega_c$ such that if $\omega\in (0,\omega_c)$, we have $d''(\omega)<0$, and if $\omega>\omega_c$, we obtain $d''(\omega)>0$. At $\omega=\omega_c$, it is clear by the continuity that $d''(\omega_c)=0$, and we obtain a frequency $\omega$ that indicates a critical value for the stability. When $\alpha \geq 5.5 $, we obtain that $d''(\omega)<0$. Despite the fact that we can not determine the orbital stability when $\omega\neq\omega_0$, since we do not know the values of ${\rm n}(\mathcal{L})$ and ${\rm z}(\mathcal{L})$, we have a prediction for the orbital stability for different values of $\omega$.\\
\indent Summarizing the analysis above for the case $\omega=\omega_0,$ we have the following statements: let $\varphi$ be the solitary wave given by $(\ref{expsolode})$ and consider $\alpha_0\approx 4.8$. Then,
\begin{enumerate}\item[(C1)] If $\alpha\in (0,\alpha_0)$, the solitary wave $\varphi$ is orbitally stable.
\item[(C2)]If $\alpha>\alpha_0$, the solitary wave $\varphi$ is orbitally unstable. 
\end{enumerate}

\indent In \cite{FabioAdemir}, the authors prove the orbital stability for the case $\alpha=2$ using an adaptation of the arguments in \cite{we}. The lack of a smooth curve of solitary waves prevented the authors from calculating the value of $d''(\omega)$ at $\omega=\frac{4}{25}$, but they determined that ${\rm n}(\mathcal{L})=1$, ${\rm z}(\mathcal{L})=2$ and exhibited a value $\chi\in H^4$ such that $\mathcal{L}_{-}\chi=\varphi$ to calculate that the inner product $(\chi,\varphi)_{L^2}$ is negative. It is important to mention that if a smooth curve $\omega\in I\mapsto \varphi$ is available and we know its behavior in a small open interval, we obtain $\chi=-\frac{d\varphi}{d\omega}$. In addition, by deriving equation $(\ref{edo})$ with respect to $\omega$, we easily obtain $\mathcal{L}_{-}\left(-\frac{d\varphi}{d\omega}\right)=\varphi$, and thus
$(\chi,\varphi)_{L^2}=-\left(\frac{d\varphi}{d\omega},\varphi\right)_{L^2}=-\frac{1}{2}\frac{d}{d\omega}\int_{\mathbb{R}}\varphi^2(x)dx.$
Therefore, if $d''(\omega)>0$, we obtain $(\chi,\varphi)_{L^2}<0$, and all the arguments in \cite{gss1}, \cite{gss2} and \cite{we} work well to obtain the orbital stability. To prove that $(\chi,\varphi)_{L^2}<0$ at $\omega_0=\frac{4}{25}$, the authors of \cite{FabioAdemir} utilized additional facts from total positivity theory and the non-standard properties of Gegenbauer polynomials (see \cite{al}). However, despite the usefulness of calculating $(\chi,\varphi)_{L^2}$ to determine orbital stability when the inner product is negative and ${\rm n}(\mathcal{L})=1$ and ${\rm z}(\mathcal{L})=2$, we can conclude the orbital instability when $(\chi,\varphi)_{L^2}>0$, because since ${\rm z}(\mathcal{L}_{-})=1$, we have that $\chi=-\frac{d\varphi}{d\omega}$ is unique such that $\mathcal{L}_{-}\chi=\varphi$. \\
\indent The authors of \cite{denis2} studied the same equation in $(\ref{NLS-equation})$ in the $N-$dimensional case
\begin{equation}
iu_t-\gamma \Delta^2u+\beta\Delta u+|u|^{2\sigma}u=0,
\end{equation}
where $\gamma>0$, $\beta\in\mathbb{R}$ and $\sigma>0$. They considered two associated constrained minimization problems: one in which they minimize the energy with a constraint on the $L^2$-norm, and the other which is similar to problem $(\ref{minP1})$. By assuming an adequate set of conditions on the parameters $\gamma$, $\beta$, and $\sigma$, they determined the existence of minimizers and some additional properties, such as their sign, symmetry, and decay at infinity, as well as their uniqueness, nondegeneracy of the associated linearized operator, and orbital stability. Some properties obtained by the solutions in \cite{denis2} are similar to those in \cite{denis}, except for orbital stability, which was not addressed in the oldest contribution. The stability result in Theorem 1.4 of \cite{denis2},  states that if $0<\sigma<\frac{4}{N}$, then: (i) the set of minimizers of the energy with fixed mass is always stable, and
(ii) if the equivalent problem $(\ref{minP1})$ for the $N$-dimensional case has a minimizer $w$ that is nondegenerate, meaning ${\rm z}(\mathcal{L})=N+1$, and if $\mathcal{L}_{-}v=w$ for some $v\in H^4$ with $(v,w)_{L^2}<0$, then the minimizer $w$ is orbitally stable, but they did not exhibit  other values of $\sigma>0$ where the inner product $(v,w)_{L^2}$ is negative, except in the case $\sigma=1$ and $N=1$ when they mentioned the results in \cite{FabioAdemir}. Comparing it with our results in $\rm{(C1)}$-$\rm{(C2)}$, we only have $(v,w)_{L^2}<0$ when $\sigma\in (0,2.4)$ and $w=\varphi$ is given by $(\ref{expsolode})$. If $\sigma>2.4$, orbital instability occurs. Moreover, our numerical experiments in Section 4 also give that for $\omega\neq \omega_0$, we have $(v,w)_{L^2}>0$ always when $\sigma>2.75$. This means that our threshold value for stability and instability is far from the value $\frac{4}{N}$ mentioned in \cite{denis2} when $N=1$.

\indent Our paper is organized as follows: In Section 2, we present the existence of solitary wave solutions for the equation $(\ref{edo})$ by using analytical and numerical approaches. Section 3 is devoted to the spectral analysis of the operator $\mathcal{L}$ in $(\ref{operator})$. The orbital stability and instability of solitary waves is presented in Section 4.\\

\textbf{Notation.} For $s\geq0$, the Sobolev space
$H^s:=H^s(\mathbb{R})$
consists of all real distributions $f$ such that
$
\|f\|^2_{H^s}:= \int_{\mathbb{R}}(1+\xi^2)^s|\hat{f}(\xi)|^2d\xi <\infty,$ where $\hat{f}$ is the Fourier transform of $f$. The space $H^s$ is a  Hilbert space with the inner product denoted by $(\cdot, \cdot)_{H^s}$. When $s=0$, the space $H^s$ is isometrically isomorphic to the space  $L^2$ and will be denoted by $L^2:=H^0$. The norm and inner product in $L^2$ will be denoted by $\|\cdot \|_{L^2}$ and $(\cdot, \cdot)_{L^2}$. For a complex function $f=f_1+if_2\equiv(f_1,f_2)$, we denote the space $\mathbb{H}^s := H^s \times H^s$ for all $s \geq 0$. Given $z=x+iy \in \mathbb{C}$, we denote $|z|=\sqrt{x^2+y^2}.$ For $s\geq0$, notation $\mathbb{H}_e^s$ indicates the space constituted by complex functions $v\in\mathbb{H}^s$ such that $v$ is even. For $p \geq 1$, we denote $L^p = L^p(\mathbb{R})$ as the space of all Lebesgue measurable real functions $f$ on $\mathbb{R}$ that satisfy $\int_{\mathbb{R}} |f(x)|^pdx < \infty$, with norm $\|f\|_{L^p} = \left( \int_{\mathbb{R}} |f(x)|^p dx \right)^{\frac{1}{p}}$. The notation $\mathbb{L}^p$ refers to the natural product space $\mathbb{L}^p = L^p \times L^p$ in the complex context. For $p=2$, we have $L^p=L^2$ as defined above.

\section{Existence of solitary waves}\label{section-existence}

In this section, we are going to prove the existence of even solitary wave solutions $\varphi$ for the equation \eqref{edo}. 

\subsection{Existence of even positive solutions via minimization problem} Let $\alpha > 0$ be fixed.  To obtain even solitary wave solutions that solve \eqref{edo}, we first use a variational approach in order to minimize a suitable  constrained functional. Indeed, let $\tau > 0$ be fixed and consider the set
$
	\mathcal{Y}_\tau := \left\{ u \in \mathbb{H}^2;\ \int_{\mathbb{R}} |u|^{\alpha + 2} dx = \tau \right\}.
$ For $\omega > 0$, we define the functional $\mathcal{B}_\omega: \mathbb{H}_e^2 \rightarrow \mathbb{R}$ given by
\begin{equation}\label{functB}
	\mathcal{B}_\omega(u) := \frac{1}{2} \int_{\mathbb{R}} |u_{xx}|^2+|u_x|^2 + \omega |u|^2 dx.
\end{equation}
\indent Using the classical concentration compactness principle (see \cite{lions1} and \cite{lions2}), we have the first result that guarantees the existence of complex solitary waves that solves $(\ref{min-problem})$ and equation $(\ref{edo})$ in the weak sense. \begin{proposition}\label{minimizationproblem}
	Let $\tau>0$ be fixed and consider $\omega > 0$. The minimization problem
	\begin{equation}\label{min-problem}
		\Gamma_{\omega,\tau}:= \inf_{u \in \mathcal{Y}_\tau} \mathcal{B}_\omega(u)
	\end{equation}
has at least one solution, that is, there exists a complex-valued  function $\Phi \in \mathcal{Y}_\tau$ such that $\mathcal{B}_\omega(\Phi) = \Gamma_{\omega,\tau}$. Moreover, $\Phi$ is an (weak) solution of the equation	
\begin{equation}\label{eq-min-problem}
	\Phi''''-\Phi'' + \omega \Phi - |\Phi|^\alpha \Phi = 0.
\end{equation}
\end{proposition}

\begin{proof} The first part follows from similar ideas as in \cite[Theorem 1.2]{denis2}, \cite[Theorem 1.1]{denis} and \cite[Section 2]{levandosky}. The second part follows from Lagrange multiplier theorem, which asserts the existence of a constant $C\in\mathbb{R}$ such that
\begin{equation}\label{eq-min-problem33}
	\Phi''''-\Phi'' + \omega \Phi - C|\Phi|^\alpha \Phi = 0.
\end{equation}
Multiplying equation $(\ref{eq-min-problem33})$ by $\overline{\Phi}$ and integrating over $\mathbb{R}$, we obtain that $2\mathcal{B}_{\omega}(\Phi)=C\tau$. This means that $C>0$, and we can assume $C=1$ (see \cite[page 629]{Cristofani} for details on how to do so). This last fact establishes $(\ref{eq-min-problem})$ and finishes the proof of the proposition.
\end{proof}

\begin{remark}\label{converg}
From $(\ref{eq-min-problem})$ we automatically conclude that $\Phi$ obtained in Proposition $\ref{minimizationproblem}$ is an element of $\mathbb{H}^{4}$. This fact enables to deduce $\Phi^{(n)}(x)\rightarrow 0$ as $|x|\rightarrow \infty$ for all $n=0,1,2,3$. If $\alpha$ is an even integer, we obtain by a standard bootstrapping argument, that $\Phi$ is smooth and $\Phi^{(n)}(x)\rightarrow 0$ as $|x|\rightarrow \infty$ for all $n\in\mathbb{N}$.
\end{remark}

\indent Next, we have the following important result that connects the explicit solution $(\ref{expsolode})$ with the minimization problem $(\ref{min-problem})$.

\begin{proposition}\label{sechmin}
Consider $\phi$ as the solution in $(\ref{expsolode})$. Then $\phi$ can be considered a minimizer of the problem $(\ref{min-problem})$.
\end{proposition}
\begin{proof}  Let $\phi$ be the solution given by $(\ref{expsolode})$. Since $\int_{\mathbb{R}} \phi^{\alpha+2}dx > 0$, we can assume in the minimization problem $(\ref{min-problem})$, given that $\omega$ and $\tau$ are arbitrary positive numbers, that $\omega = \omega_0$, where $\omega_0$ is given by $(\ref{omega})$, and $\tau=\tau_0 > 0$ is defined by $\tau_0 = \int_{\mathbb{R}} \phi^{\alpha+2}dx > 0$. By Proposition $\ref{minimizationproblem}$, we obtain a minimizer $\Phi$ for the functional $\mathcal{B}_{\omega_0}$ in $(\ref{functB})$, and $\Phi$ solves equation $(\ref{eq-min-problem})$. Multiplying equation $(\ref{eq-min-problem})$ by $\overline{\Phi}$ and integrating the result over $\mathbb{R}$, we obtain 
\begin{equation}\label{mingamma1}\Gamma_{\omega_0,\tau_0}=\mathcal{B}_{\omega_0}(\Phi)=\frac{\tau_0}{2}.\end{equation}
\indent On the other hand, since $\phi$ is a solution of $(\ref{edo})$, we can multiply equation $(\ref{edo})$ by $\phi$ and integrate the result over $\mathbb{R}$ to obtain 
\begin{equation}\label{mingamma2}\mathcal{B}_{\omega_0}(\phi)=\frac{1}{2}\int_{\mathbb{R}}\phi^{\alpha+2}dx=\frac{\tau_0}{2}.\end{equation}
Comparing $(\ref{mingamma1})$ and $(\ref{mingamma2})$, we obtain that $\phi$ can be considered a minimizer of the functional $\mathcal{B}_{\omega_0}$ as required, and we denote it as $\phi = \varphi$ henceforth to simplify the notation.
\end{proof}

\begin{remark}
    Let $s \in (0,1]$ be fixed, and consider $g^{*} : \mathbb{R} \rightarrow [0, +\infty)$ as the symmetric-decreasing rearrangement associated with the function $g : \mathbb{R} \rightarrow \mathbb{R}$. The function $g^{*}$ can be considered even, real-valued over the entire real line, and it has a single-lobe profile, meaning it has only one maximum point, which can be taken at $x = 0$. Such special  functions also arise as minimizers in problems similar to $(\ref{min-problem})$  as an application of the P\'olya-Szeg\"o inequality
\begin{equation}\label{polya}
    ||(-\Delta)^s g^{*}||_{L^2} \leq ||(-\Delta)^s g||_{L^2},
\end{equation}
and the fact that the norm in $L^p$ is preserved by symmetric-decreasing rearrangements (see, for instance, \cite[Proposition 3.6]{moraes}). Here, $(-\Delta)^s$ denotes the standard fractional Laplacian operator defined by
$
    \widehat{(-\Delta)^s f}(\xi) = |\xi|^{2s} \widehat{f}(\xi),
$
where $\xi \in \mathbb{R}$. As far as we can see, the same P\'olya-Szeg\"o inequality in $(\ref{polya})$ fails for symmetric-decreasing rearrangements if $s > 1$, and therefore has very limited use when dealing with higher-order derivatives (see \cite[Pages 3 and 4]{lezsok} for more details). This fact prevents us to use the useful inequality in $(\ref{polya})$ to obtain real-valued functions with single-lobe profile that solves the minimization problem in $(\ref{min-problem}).$

\end{remark}

\begin{remark}\label{remarkeven}
      If $\alpha$ is an even positive integer, the authors deduce in \cite[Theorem 3]{lezsok} (see also \cite{denis2} and \cite{denis}) that a minimizer $\varphi$ obtained in Proposition \ref{minimizationproblem} can be expressed as $\varphi = e^{i\theta_0}\varphi^{\#}(\cdot - x_0)$ for some $\theta_0 \in \mathbb{R}$ and $x_0 \in \mathbb{R}$, where $\varphi$ is a real-valued, even function. The function $\varphi^{\#}$ is defined as the Fourier rearrangement of the function $\varphi \in L^2$, and it is given by $f^{\#} = \mathcal{F}^{-1}\left\{(\mathcal{F}f)^{*}\right\}$, where $\mathcal{F}g=\widehat{g}$ denotes the standard Fourier transform in $\mathbb{L}^2$ and $\mathcal{F}^{-1}$ its inverse defined in the same space. Furthermore, the authors claim that for other values of $\alpha$, they do not know whether the minimizer can be expressed in the form $\varphi = e^{i\theta_0}\varphi^{\#}(\cdot - x_0)$. According to our best knowledge, the main problem lies in determining whether the $\mathbb{L}^p$-norm is preserved under Fourier rearrangement, or at least whether $||f||_{\mathbb{L}^p}\leq ||f^{\#}||_{\mathbb{L}^p}$. In the first case, we obtain a property that holds true when dealing with symmetric-decreasing rearrangements (see \cite[Pages 18 and 19]{lezsok} for a more detailed explanation of this important topic).
\end{remark}

\begin{proposition}\label{propunique}  Let $\varphi$ be a real minimizer obtained in Remark $\ref{remarkeven}$. We obtain that ${\rm n}(\mathcal{L}_{-})=1$ and ${\rm n}(\mathcal{L}_{+})=0.$ In particular, for the explicit solitary wave $\varphi$ in $(\ref{expsolode})$, we also have the same conclusion.
\end{proposition}

\begin{proof}
	For $\omega>0$, we have to notice that the operator $\mathcal{L}$ can be obtained by defining the functional $G_{\omega}(u) = E(u) + \omega F(u)$ where $E$ and $F$ are conserved quantities associated to equation $(\ref{NLS-equation})$ given respectively by $(\ref{E})$ and $(\ref{F})$.\\
 \indent By $(\ref{edo}$) we have that $G_{\omega}'(\varphi,0) = 0$, that is, $(\varphi,0)$ is a critical point of $G$. In addition, we have that $G_{\omega}''(\varphi,0) = \mathcal{L}$, where $\mathcal{L}$ is the linear operator given by $(\ref{operator})$. Since $\varphi$ is a minimizer of the problem $(\ref{min-problem})$, we obtain by the min-max principle (see \cite[Theorem XIII.2]{ReedSimon}) that ${\rm n}(\mathcal{L}) \leq 1$. Moreover, using again the equation \eqref{edo}, we get
$
    (\mathcal{L}_{-} \varphi, \varphi)_{L^2_{per}} = - \alpha \int_{\mathbb{R}} |\varphi|^{\alpha+2} dx  = - \alpha \tau < 0,
$
that is, ${\rm n}(\mathcal{L}_{-}) \geq 1$, where $\mathcal{L}_{-}$ is the linear operator as in $(\ref{L1L2})$. This last fact enables us to conclude ${\rm n}(\mathcal{L}) = 1$ and since $\mathcal{L}$ has a diagonal form, we automatically obtain
$
    {\rm n}(\mathcal{L}_{-}) = 1$ and ${\rm n}(\mathcal{L}_{+}) = 0.
$
\end{proof}

\subsection{ Generation of solitary wave  solutions via a numerical approach}

In this subsection, we propose Petviashvili's iteration method to generate the standing solitary wave solutions of \eqref{NLS-equation}.
Setting $u(x,t) = e^{i\omega t} \varphi(x),~~~~\omega>0$ in \eqref{NLS-equation}   we obtain \eqref{edo}.  Then,  applying the Fourier transform  to \eqref{edo} yields
\begin{equation}\label{iterate1}
\left(  \xi^{4} +\xi^{2}+ \omega  \right) \widehat{\varphi}(\xi)-  \widehat{|\varphi|^{\alpha}\varphi}(\xi)=0, \quad \xi \in \mathbb{Z}.
\end{equation}
An iterative algorithm for numerical calculation of $\widehat{\varphi}(\xi)$ for the equation \eqref{iterate1} can be proposed in the form
\begin{equation}\label{iterative1}
\widehat{\varphi}_{n+1}(\xi)=\frac{  \widehat{|\varphi_n|^{\alpha}\varphi_n}(\xi)  }
{\xi^{4} +\xi^{2}+ \omega},\hspace{30pt} n\in\mathbb{N}
\end{equation}
where $\widehat{\varphi}_n(\xi)$ is the Fourier transform of  the $n^{\text{th}}$ iteration ${\varphi}_n$ of the numerical solution. 
\noindent Since the above algorithm is usually divergent, the stabilizing factor is defined  as
\begin{equation}\label{sf}
  M_n=\frac{\big(   (\partial_x^4-\partial_x^2+\omega)\varphi_n, \varphi_n   \big)_{L^2}}{(|\varphi_n|^\alpha\varphi_n, \varphi_n)_{L^2}}.
\end{equation}

\indent The Petviashvilli's method is then given by
\begin{equation}\label{iterative2}
\widehat{\varphi}_{n+1}(\xi)=\frac{(M_n)^{\nu}} {\xi^{4} +\xi^{2}+ \omega}  \widehat{|\varphi_n|^{\alpha}\varphi_n}(\xi),
\end{equation}
 where the free parameter $\nu$ is chosen as $(\alpha+2)/(\alpha+1)$ for the fastest convergence.
\noindent
The iterative process is controlled by the error between two consecutive iterations given by
$$
  \mbox{Error}(n)=\|\varphi_n-\varphi_{n-1}\|_{L^{\infty}}
$$
and the stabilization factor error given by
$ |1-M_n|. $ The residual error is determined by
$$
{RES(n)}= \|{ \mathcal{T}} \varphi_n\|_{L^{\infty}},$$ where  
$${ \mathcal{T}}\varphi= \varphi''''-\varphi'' + \omega \varphi - |\varphi|^\alpha \varphi. $$

\begin{figure}[h]
 \begin{minipage}[t]{0.4\linewidth}
  \includegraphics[width=3.2in]{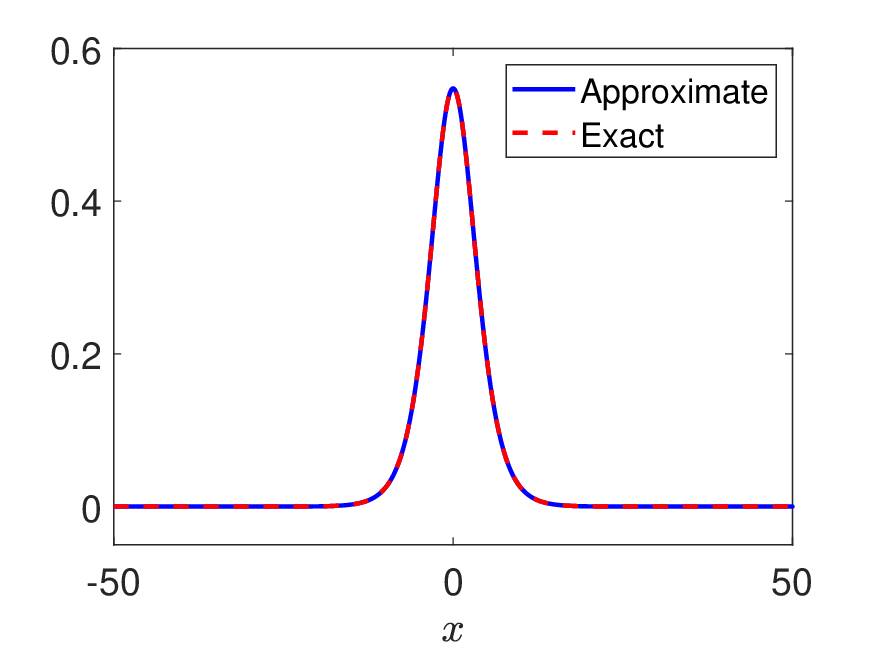}
 \end{minipage}
\hspace{30pt}
\begin{minipage}[t]{0.40\linewidth}
   \includegraphics[width=3.2in]{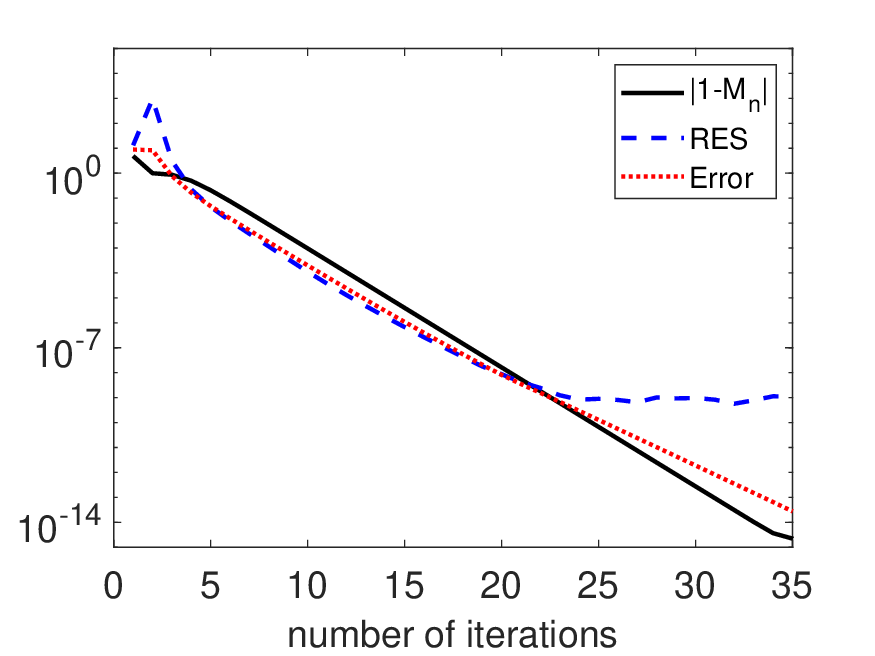}
 \end{minipage}
  \caption{Comparison of the exact solution \eqref{expsolode} with the numerical solution of \eqref{edo} for  $\omega=\omega_0=4/25$, $\alpha=2$ (left), and the variation of  \mbox{Error$(n)$}, $|1-M_n|$ and $RES(n)$ with the number of iterations in semi-log scale (right).}
 \label{exact_p2}
\end{figure}

\begin{figure}[h]
 \begin{minipage}[t]{0.4\linewidth}
  \includegraphics[width=3.2in]{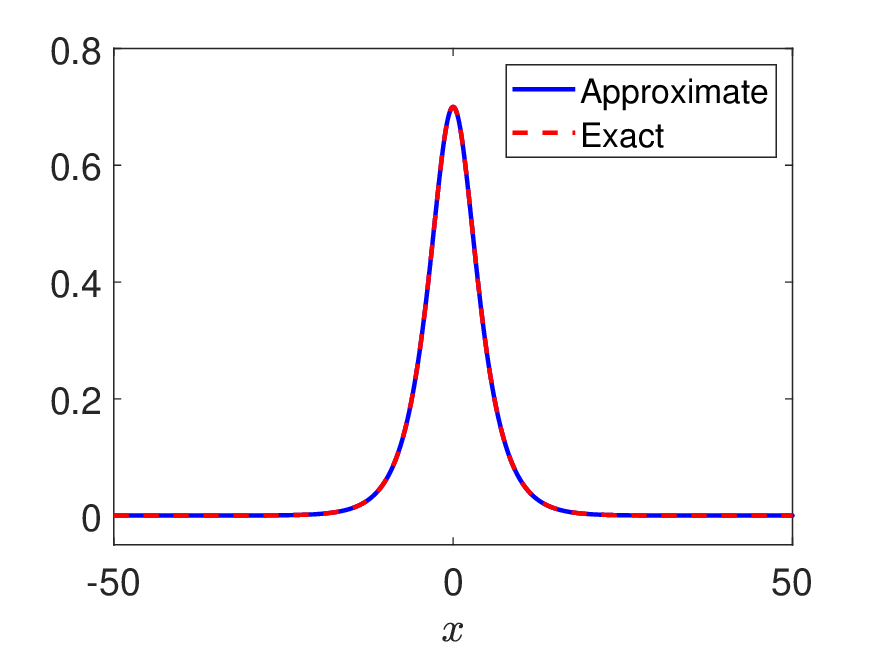}
 \end{minipage}
\hspace{30pt}
\begin{minipage}[t]{0.40\linewidth}
   \includegraphics[width=3.2in]{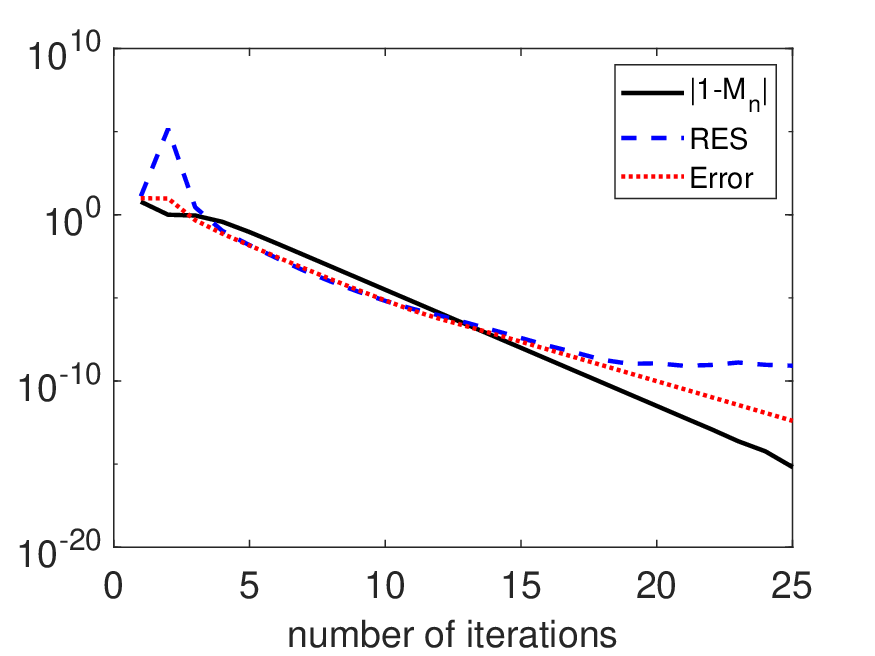}
 \end{minipage}
  \caption{Comparison of the exact solution \eqref{expsolode} with the numerical solution of \eqref{edo} for  $\omega=\omega_0=9/100$, $\alpha=4$ (left), and the variation of  \mbox{Error$(n)$}, $|1-M_n|$ and $RES(n)$ with the number of iterations in semi-log scale (right).}
 \label{exact_p4}
\end{figure}
  
\indent The expression in \eqref{expsolode}  gives the exact solution of \eqref{edo} for each $\alpha$ and a fixed $\omega_0$.
In order to test the accuracy of our scheme, we compare the exact solution \eqref{expsolode} with the numerical solution. We choose the space interval as $ [-200,200] $ and the number of grid points as $N=2^{13}$. In the left panels of Figure \ref{exact_p2} and Figure \ref{exact_p4}, we show the comparison of the exact solution \eqref{expsolode} with the numerical solution of \eqref{edo} for  $\omega=\omega_0=4/25$, $\alpha=2$ and for  $\omega=\omega_0=9/100$, $\alpha=4$, respectively. The right panels of the figures present the variation of  \mbox{Error$(n)$}, $|1-M_n|$ and $RES(n)$ with the number of iterations in semi-log scale for the same values of $\alpha$ and $\omega$. We also observe that the $L_\infty$ error between the exact and numerical solution is of order $10^{-14}$.  These results show that our numerical scheme captures the solution quite well.

\begin{figure}[h]
 \begin{minipage}[t]{0.4\linewidth}
  \includegraphics[width=3.2in]{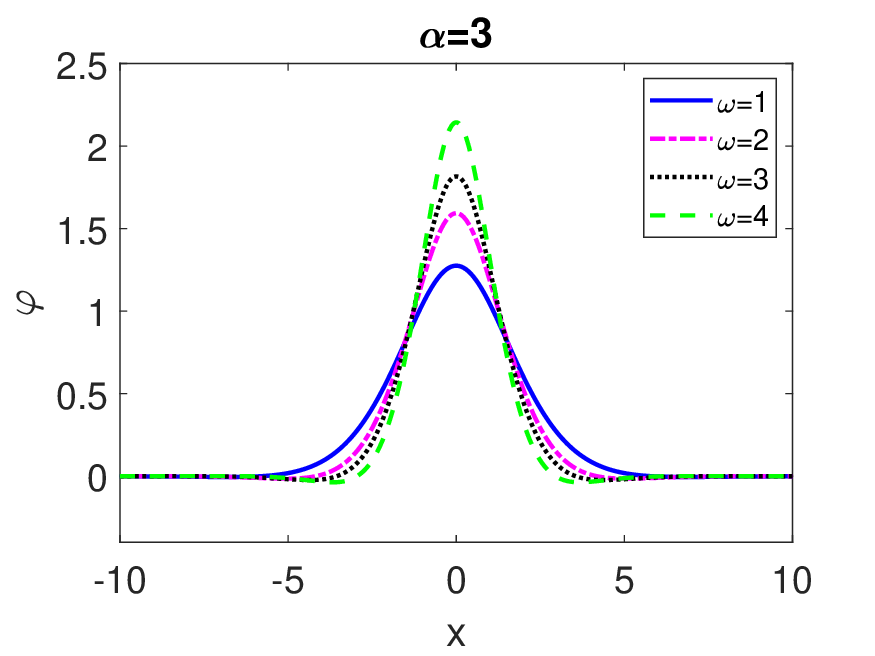}
 \end{minipage}
\hspace{30pt}
\begin{minipage}[t]{0.40\linewidth}
   \includegraphics[width=3.2in]{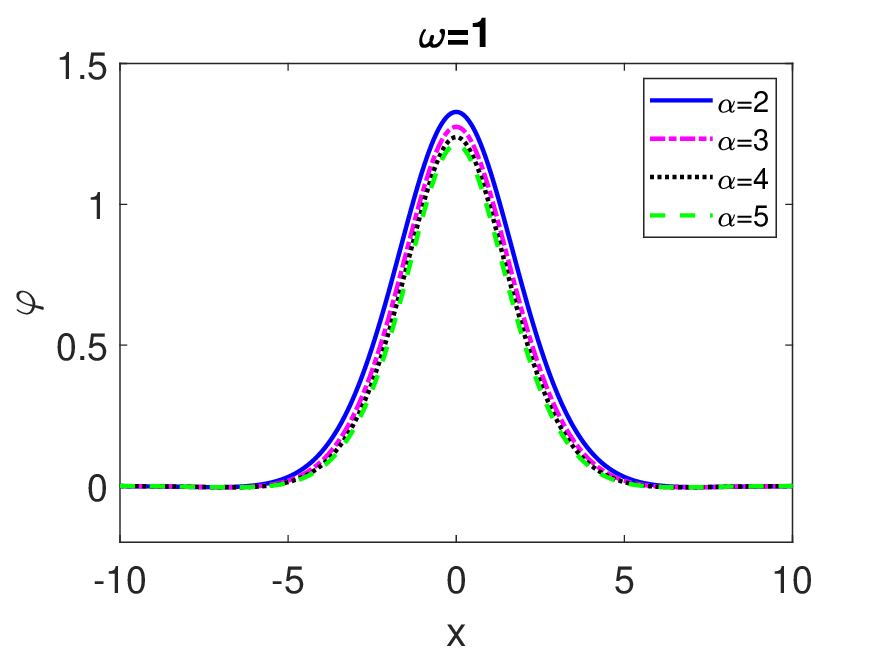}
 \end{minipage}
  \caption{Numerical wave profiles for various values of $\omega$, with  $\alpha=3$ (left) and numerical wave profiles for various values of $\alpha$ with wave frequency $\omega=1$ (right).}
 \label{wave_profiles}
\end{figure}



Now, we generate solutions numerically for several values of $\alpha$ and $\omega \neq \omega_0$ where the exact solution does not exist. 
The left panel of Figure \ref{wave_profiles} shows the profiles of the solutions for a fixed value of $\alpha$ and various values of $\omega$. In the right panel, we present the profiles for a fixed value of $\omega$ and several values of $\alpha$.
 It is known that the fourth-order ODE \eqref{edo} is identical to the ODEs satisfied by the traveling wave solutions of both fifth-order KdV (\cite{PC}) and Kawahara (\cite{sprenger}) equations if the nonlinearities are the same.  According to \cite{PC} and \cite{sprenger} the solutions have decaying oscillatory tails for $\omega > 0.25$. Performing lots of numerical experiments for several values of $\alpha$ and $\omega$, we observe the same result numerically.  Even if they are not visible on some plots of Figure \ref{wave_profiles}, the oscillations still exist when $\omega > 0.25$. Here, we point out that the $\omega_0$ for any $\alpha$ in \eqref{omega} corresponding to the exact solution is already less than $0.25$.


\begin{remark}
The equation $(\ref{edo})$ yields non-positive solutions for  \mbox{$\omega>0.25$.} This point is very delicate compared with similar results in the current literature, as it is well known for the solutions of the standard NLS equation $-\varphi''+\omega\varphi-|\varphi|^{\alpha}\varphi=0$ where all smooth solutions that decay at infinity are positive.
\end{remark}

\begin{remark}
As reported in \cite[Theorem 1.4]{denis}, solitary waves that change their sign may occur if $\omega>0.25$. Besides, the solitary waves solve the minimization problem $(\ref{min-problem})$; they also arise as solutions of the problem
$
\nu_\omega:= \inf_{u\in \mathbb{H}^2} G_{\omega}(u),
$ where $G_{\omega}(u)=E(u)+\omega F(u)$. Our experiments somehow match with the scenario in \cite[Theorem 1.4]{denis} since the solutions of $(\ref{edo})$ are non-positive when $\omega > 0.25$ for each $\alpha>0$.
\end{remark}
\section{Spectral analysis}\label{section-spectral}
\subsection{Brief review of total positivity theory}\label{reviewsec} Our intention is to study the spectrum of the linear operator $\mathcal{L}$ in $(\ref{operator})$ and to do so, we need to consider some tools of the total positivity theory established in \cite{al}, \cite{albo}, and
\cite{ka}. Let $\mathcal{S}$ be the operator defined on
a dense subspace of $L^2$ by
\begin{equation}\label{opT}
\mathcal{S}g(x)=Mg(x)+ \omega g(x)-(\alpha+1)\psi^{\alpha}(x)g(x),
\end{equation}
where $\alpha\geq1$ is real number, $\omega>0$ is a real parameter, $\psi$ is an even real-valued positive solution of the general equation
\begin{equation}\label{solitary}
(M+\omega)\psi=\psi^{\alpha+1}.
\end{equation}
\indent We also assume that function $\psi$ has a suitable decay at infinity and $M$ is defined as a Fourier multiplier operator by
$
\widehat{Mg}(\xi)=m(\xi)\widehat{g}(\xi),
$
where $\xi\in\mathbb{R}$. Here, the circumflexes notation in the last equality denotes the Fourier transform, $m(\xi)$ is a measurable,
locally bounded, even function defined on $\R$, and satisfying the following assumptions:\\
    {\rm (i)} $A_1|\xi|^\mu\leq m(\xi)\leq A_2|\xi|^\mu$ for $|\xi|\geq
    \xi_0$,\\
    {\rm (ii)} $m(\xi)\geq 0$,\\
\noindent where $A_1$, $A_2$, and $\xi_0$ are positive real constants, and $\mu\geq1$. Under the above assumptions we
have the following basic result.

\begin{lemma}\label{gelema}
The operator $\mathcal{S}$ is a closed, unbounded, self-adjoint
operator on $L^2$ whose spectrum consists of the interval
$[\omega,\infty)$ together with a finite number of discrete
eigenvalues in the interval $(-\infty,\omega]$, in which all of them
have finite  multiplicity. In addition, zero is an
eigenvalue of $\mathcal{S}$ with eigenfunction  $\psi'$.
\end{lemma}
\begin{proof}
See Proposition 1 in \cite{abh}.
\end{proof}

Next, our intention is to obtain suitable spectral properties for the linear operator $\mathcal{S}$ in $(\ref{opT})$. To this end, let us introduce the family of operators $\mathcal{P}_{\kappa}$, $\kappa\geq0$, defined on $L^2$ as
$
\mathcal{P}_{\kappa} g(\xi)=
\frac{1}{w_{\kappa}(\xi)}\int_{\R}K(\xi-y)g(y)dy,
$
where $K(\xi)=(\alpha+1)\widehat{\psi^{\alpha}}(\xi)$ and $w_{\kappa}(\xi)=m(\xi)+\kappa+\omega$.
The family of operators $\mathcal{P}_{\kappa}$ act over the  Hilbert space
$
X=\left\{g\in L^2;\,
\|g\|_{X,\kappa}=\left(\int_{\R}|g(\xi)|^2w_\kappa(\xi)d\xi\right)^{1/2}<\infty\right\}.
$

By \cite[Proposition 2.2]{al}, we see that $\mathcal{P}_{\kappa}$ is a compact, self-adjoint operator on
$X$ with respect to the norm $\|\cdot\|_{X,\kappa}$. Using the spectral theorem for compact and self-adjoint operators, we conclude that $\mathcal{P}_{\kappa}$ has a family of eigenfunctions
$\{\chi_{i,\kappa}\}_{i=0}^\infty$ forming  an orthonormal basis
of $X$. Moreover, the corresponding eigenvalues
$\{\lambda_i(\kappa)\}_{i=0}^\infty$ are real and can be numbered in
order of decreasing absolute value
$|\lambda_0(\kappa)|\geq|\lambda_1(\kappa)|\geq\ldots\geq0.$\\
\indent Let us recall some results of \cite{al} and \cite{albo}. The first
one is concerned with an equivalent formulation of the eigenvalue
problem associated with $\mathcal{S}$.

\begin{lemma}\label{multilema}
Suppose $\kappa\geq0$. Then $-\kappa$ is an eigenvalue of
$\mathcal{S}$ (as an operator on $L^2$) with eigenfunction $g$
if, and only if, $1$ is an eigenvalue of $\mathcal{P}_{\kappa}$ (as
an operator on $X$) with eigenfunction $\widehat{g}$. In particular,
both eigenvalues have the same multiplicity.
\end{lemma}
\begin{proof}
See Corollary 2.3 in \cite{al}.
\end{proof}

Since $\mathcal{P}_{\kappa}$ is a compact operator, we have a Krein-Rutman theorem given by the next result.

\begin{lemma}\label{simplelema}
The eigenvalue $\lambda_0(0)$ of $\mathcal{P}_{0}$ is positive,
simple, and has a strictly positive eigenfunction $\chi_{0,0}$.
Moreover, $\lambda_0(0)>|\lambda_1(0)|$.
\end{lemma}
\begin{proof}
See Lemma 8 in \cite{albo}.
\end{proof}

\begin{definition} A function $f:\R\to\R$ belongs to the class PF(2)
if:\\
{\rm (i)} $f(x)>0$ for all $x\in\R$,\\
{\rm (ii)} for any $x_1,x_2,y_1,y_2\in\R$ with $x_1<x_2$ and $y_1<y_2$, it follows that
$f(x_1-y_1)f(x_2-y_2)-f(x_1-y_2)f(x_2-y_1)\geq0,$\\
{\rm (iii)} strict inequality holds in {\rm (ii)} whenever the intervals $(x_1,x_2)$ and $(y_1,y_2)$ intersect.
\end{definition}
A sufficient condition for a smooth function to be in class $PF(2)$ is that it is logarithmically concave. More precisely, we have:

\begin{lemma}\label{logcon}
Let $f:\mathbb{R}\to\mathbb{R}$ be a positive and twice differentiable function. If $(\log f(x))''<0$ for all $x\neq0$, we have that $f$ belongs to the class PF(2).
\end{lemma}
\begin{proof}
See Lemma 4.3 in \cite{al} or Lemma 10 in \cite{albo}.
\end{proof}

\begin{lemma}\label{logconvolution}
Let $f_i:\mathbb{R}\to\mathbb{R}$ be positive and twice differentiable functions, where $i=1,2,\cdots, n$. If $f_i$ is a logarithmically concave function for each $i$, then the finite convolution $f_1\ast f_2\ast\cdots \ast f_n$ is also a logarithmically concave function. 
\end{lemma}
\begin{proof}
See \cite[Section 3]{bl}.
\end{proof}
The main theorem in \cite{al} reads as follows.

\begin{theorem}\label{mainal}
Suppose $\widehat{\psi}>0$ on $\R$ and $K=(\alpha+1)\widehat{\psi^{\alpha}}\in
PF(2)$. Then $\mathcal{S}$ satisfies the following conditions:\\
{\rm (P1)} $\mathcal{S}$ has a simple, negative eigenvalue $\kappa$,\\
{\rm (P2)} $\mathcal{S}$ has no negative eigenvalue other than $\kappa$,\\
{\rm (P3)} the eigenvalue $0$ of $\mathcal{S}$ is simple.
\end{theorem}
\begin{proof}
See Theorem 3.2 in \cite{al}.
\end{proof}

\subsection{Spectral analysis for the operator $\mathcal{L}$} In this subsection, we use the total positivity theory determined in previous subsection to get the spectral properties for the general operator $(\ref{operator})$. In fact, as before, since $\mathcal{L}$ is a diagonal operator, its eigenvalues are given by the eigenvalues of the operators $\mathcal{L}_{-}$ and $\mathcal{L}_{+}$ in $(\ref{operator})$.

\subsubsection{The spectrum of $\mathcal{L}_{-}$}

The operator $\mathcal{L}_{-}$ at the point $\omega=\omega_0$ then reads as
\begin{equation}\label{L1m}
\mathcal{L}_{-}= \partial_x^4-\partial_x^2+\omega_0-(\alpha+1)\varphi^{\alpha},
\end{equation}
where $\varphi=\phi$ is the solitary wave solution with hyperbolic secant profile given by $(\ref{expsolode})$ which is also a minimizer of the problem $(\ref{min-problem})$.

We now can prove the following result.

\begin{proposition}\label{specL1}
The operator $\Lum$ in \eqref{L1m} defined on $L^2$ with domain $H^4$
has a unique negative eigenvalue, which is simple with positive associated eigenfunction. The eigenvalue zero is
simple with associated eigenfunction $\varphi'$. Moreover the rest of the
spectrum is bounded away from zero and the essential spectrum is the interval
$[\omega_0,\infty)$.
\end{proposition}
\begin{proof}
 By \eqref{edo} we see that $\varphi'$ is an eigenfunction of  $\mathcal{L}_{-}$ associated with the eigenvalue zero. Also, by Lemma \ref{gelema}, the essential spectrum is exactly $[\omega_0,\infty)$. We need to obtain properties {\rm (P1)}-{\rm (P3)}. In fact, first of all, we see that $\mathcal{L}_{-}$ is an operator of the form \eqref{opT} with $m(\xi)=\xi^4+\xi^2$ and $\psi=\varphi$. Thus, it is easy to see that
$
|\xi|^4\leq m(\xi)\leq 2|\xi|^4$ for all $|\xi|\geq1.
$
Hence, the operator $\Lum$ satisfies the assumptions in Subsection
\ref{reviewsec} with $A_1=1$, $A_2=2$, $\mu=4$, and $\xi_0=1$. In view
of Lemma \ref{gelema} and Theorem \ref{mainal} it suffices to prove that
$\widehat{\varphi}>0$ and $\widehat{\varphi^{\alpha}}\in PF(2)$. In fact, by \cite[Page 33, Formula 7.5]{fritz}, we first obtain  
\begin{equation}\label{fouriersol}\begin{array}{lllll}
\widehat{\varphi}(\xi)&=&\displaystyle a_0(\sech^{\frac{4}{\alpha}}(b_0\cdot))^\wedge(\xi)=2^{\frac{4}{\alpha}-2}\frac{a_0}{b_0}\Gamma\left(\frac{4}{\alpha}\right)^{-1}\Gamma\left(\frac{2}{\alpha}+i\frac{\xi}{2b_0}\right)\Gamma\left(\frac{2}{\alpha}-i\frac{\xi}{2b_0}\right)\\\\
&=&\displaystyle 2^{\frac{4}{\alpha}-2}\frac{a_0}{b_0}\Gamma\left(\frac{4}{\alpha}\right)^{-1}\left|\Gamma\left(\frac{2}{\alpha}+i\frac{\xi}{2b_0}\right)\right|^2,
\end{array}
\end{equation}
where $\Gamma$ denotes the usual Gamma function defined by $\Gamma(z)=\int_0^{\infty}t^{z-1}e^{-t}dt$ for ${\rm Re}(z)>0$. In addition, in the last equality in $(\ref{fouriersol})$ we have used the property $\Gamma(z)\Gamma(\bar{z})=\Gamma(z)\overline{\Gamma(z)}=|\Gamma(z)|^2$. Formula $(\ref{fouriersol})$ gives us automatically that $\widehat{\varphi}>0$ for all $\alpha>0$.\\ 
 \indent On the other hand, since $\varphi^{\alpha}(x)=a_0^{\alpha}\sech^{4}(b_0x)$, we obtain by using a similar argument as determined in $(\ref{fouriersol})$ 
 \begin{equation}\label{fouriesol1}\begin{array}{llll}\displaystyle \widehat{\varphi^{\alpha}}(\xi)&=&\displaystyle a_0^{\alpha}(\sech^{4}(b_0\cdot))^{\wedge}(\xi)=6\frac{a_0^{\alpha}}{b_0}\Gamma(4)^{-1}\Gamma\left(2+i\frac{\xi}{2b_0}\right)\Gamma\left(2-i\frac{\xi}{2b_0}\right)\\\\
 &=&\displaystyle \frac{2a_0^{\alpha}}{3b_0}\Gamma\left(2\left(1+i\frac{\sqrt{\alpha^2+4\alpha+8}}{2\alpha}\xi\right)\right)\Gamma\left(2\left(1-i\frac{\sqrt{\alpha^2+4\alpha+8}}{2\alpha}\xi\right)\right).
 \end{array}\end{equation}
We can obtain a convenient expression for $(\ref{fouriesol1})$. To do so, we need to use some properties concerning Gamma function such as $\Gamma(z+1)=z\Gamma(z)$, as well as, non-standard expressions for $\Gamma(2z)$, $\Gamma(1+iy)\Gamma(1-iy)$, and $\Gamma\left(\frac{1}{2}+iy\right)\Gamma\left(\frac{1}{2}-iy\right)$ as found in \cite[Page 236, Formulas 6.1.18, 6.1.30 and 6.1.31]{abrasteg}. Thus, from $(\ref{fouriesol1})$, we obtain the final and simplified expression
\begin{equation}\label{fouriesol2} \widehat{\varphi^{\alpha}}(\xi)=\frac{a_0^{\alpha}}{3b_0^2}\left(\frac{1}{4}+\frac{\alpha^2+4\alpha+8}{4\alpha^2}\xi^2\right)\frac{\pi}{\cosh\left(\frac{\pi}{4b_0}\xi\right)}\frac{\xi}{\sinh\left(\frac{\pi}{4b_0}\xi\right)}.
\end{equation}
 The expression $(\ref{fouriesol2})$ gives us that $\widehat{\varphi^{\alpha}}>0$. In addition, let us consider the function \mbox{$\beta(x)=\sech(b_0x)$.} Using the arguments in \cite[Page 33, Formula 7.1]{fritz}, we deduce  $\widehat{\beta}(\xi)=\frac{\pi}{2b_0}\sech\left(\frac{\pi\xi}{2b_0}\right)$ and $\widehat{\beta}>0$ belongs to the class $PF(2)$, since
 $\frac{d^2}{d\xi^2}\log(\widehat{\beta}(\xi))=-\frac{\pi^2}{4b_0^2}\sech^2\left(\frac{\pi\xi}{2b_0}\right)<0.$
Thus, from Lemma $\ref{logconvolution}$, we have that $ \widehat{\varphi^{\alpha}}$ belongs to the class $PF(2)$,
that is, $K=(\alpha+1)\widehat{\varphi^{\alpha}}\in PF(2)$ as desired. 
\end{proof}

\begin{remark}
It is important to highlight that the function $\varphi$ in $(\ref{expsolode})$ represents a minimizer of the problem in $(\ref{minimizationproblem})$. Therefore, one can obtain property {\rm (P1)} in Proposition $\ref{specL1}$ by applying Proposition $\ref{propunique}$.
\end{remark}

\subsubsection{The spectrum of $\mathcal{L}_{+}$} By Proposition $\ref{specL1}$, we have determined that ${\rm n}(\mathcal{L}_{-})=1$. However, we need to calculate both ${\rm n}(\mathcal{L}_{+})$ and ${\rm z}(\mathcal{L}_{+})$ in order to conclude that ${\rm n}(\mathcal{L})=1$ and ${\rm z}(\mathcal{L})=2$. To do so, we have the following result:
\begin{proposition}\label{specL2}
The operator $\mathcal{L}_{+}=\partial_x^4-\partial_x^2+\omega-\varphi^{\alpha}$  defined on $L^2$ with domain $H^4$
has no negative eigenvalues. The eigenvalue zero is
simple with associated eigenfunction $\varphi$. Moreover the rest of the
spectrum is bounded away from zero and the essential spectrum is the interval
$[\omega_0,\infty)$.
\end{proposition}
\begin{proof}
We need to prove that $0$ is a simple eigenvalue of $\mathcal{L}_{+}$ and ${\rm n}(\mathcal{L}_{+})=0$. This result, in fact, is not explicitly covered by the arguments in Subsection 3.1, but we can still adapt it to the linear operator $\mathcal{L}_{+}$. Indeed, the essential spectrum is $[\omega_0,\infty)$ and this fact is due to Lemma $\ref{gelema}$. In view of $(\ref{edo})$, it is clear that zero is an eigenvalue with eigenfunction $\varphi$. Thus, Lemma \ref{multilema} implies that $1$ is an
eigenvalue of $\mathcal{P}_{0}$ with eigenfunction $\widehat{\varphi}$. We claim that: if $\lambda_0(0)$ is the first eigenvalue of $\mathcal{P}_{0}$, we have  $\lambda_0(0)=1$. In fact, let us assume by contradiction that $\lambda_0(0)$ is the first eigenvalue of $\mathcal{P}_{0}$ and $\lambda_0(0)\neq1$. Consider $\chi_{0,0}$
be the eigenfunction associated to $\lambda_0(0)$. Using Lemma \ref{simplelema}, we have
$\chi_{0,0}>0$. Since $\mathcal{L}_{+}\varphi=0$ and $\widehat{\varphi}>0$, by Proposition $\ref{specL1}$, the inner product in $L^2$ between $\chi_{0,0}$ and $\widehat{\varphi}$ is then strictly positive, which contradicts the fact that
the eigenfunctions are orthogonal. Gathering the above claim and Lemma \ref{simplelema}, we obtain that 1 is
a simple eigenvalue of $\mathcal{P}_{0}$. Therefore, it follows from
Lemma \ref{multilema} that zero is a simple eigenvalue of $\mathcal{L}_{+}$.\\
\indent It remains to show that $\mathcal{L}_{+}$ has no negative eigenvalues. To do
so, it is sufficient, from Lemma \ref{multilema}, to prove that $1$ is
not an eigenvalue of $\mathcal{P}_{\kappa}$ for any $\kappa>0$. We
already know that
$
\lim_{\kappa\to\infty}\lambda_0(\kappa)=0
$
and $\kappa\in[0,\infty)\mapsto\lambda_0(\kappa)$ is a strictly decreasing
function (see \cite[page 9]{al}). Thus, for $\kappa>0$ and $i\geq1$,
$
|\lambda_i(\kappa)|\leq \lambda_0(\kappa)<\lambda_0(0)=1.
$
 This obviously implies that 1 cannot be an eigenvalue of $\mathcal{P}_{\kappa}$,
 $\kappa>0$, and the proof of the proposition is completed.

\end{proof}

\begin{remark}
Since $\varphi$ in $(\ref{expsolode})$ is a minimizer of the problem $(\ref{minimizationproblem})$, we could employ Propostion $\ref{propunique}$ to show, in Proposition $\ref{specL2}$, that $\mathcal{L}_{+}$ has no negative eigenvalues.
\end{remark}
\subsection{The spectrum of $\mathcal{L}$}


We state the spectral properties of the linearized operator $\mathcal{L}$ at $\omega=\omega_0$ and for $\varphi=\phi$ given by $(\ref{expsolode}).$ To prove next proposition, we use the results contained in Propostions $\ref{specL1}$ and $\ref{specL2}$.

\begin{proposition}\label{propmain}
The operator $\mathcal{L}$ in \eqref{operator} defined on $\mathbb{L}^2$
with domain $\mathbb{H}^4$ has a unique negative eigenvalue, which
is simple. The eigenvalue zero is double with associated eigenfunctions
$(\varphi',0)$ and $(0,\varphi)$. Moreover the essential spectrum is the interval
$[\omega_0,\infty)$.

\end{proposition}

\begin{flushright}
$\blacksquare$
\end{flushright}
\section{Orbital stability and instability of solitary waves}

\subsection{Local and global well-posedness}
Let us recall some important results concerning the local and global well-posedness for the Cauchy problem associated to the equation $(\ref{NLS-equation})$ given by

\begin{equation}\label{CauchyNLS}
	\left\{\begin{array}{llll}i u_t + u_{xx}-u_{xxxx} + |u|^\alpha u = 0,\ \ \ \ \ \ \mbox{in}\ \mathbb{R}\times \mathbb{R},\\\\
 u(x,0)=u_0(x),\ \ \ \ \ \ \ \ \ \ \ \ \ \ \ \ \ \ \ \ \ \ \ \ \ \mbox{in}\ \mathbb{R}.
 \end{array}\right.\end{equation}

\begin{proposition}\label{localteo}
Given $u_0\in \mathbb{H}^2$, there exists $t_0 > 0$ and a unique solution
$u\in C([0,t_0];\mathbb{H}^2)$ of \eqref{CauchyNLS} such that $u(0) = u_0$.
The solution conserves the energy and the mass in the sense that
$
E(u(t)) = E(u_0)$ and $F(u(t)) = F(u_0),$ $t\in
[0, t_0],
$
where  $E$ is defined in \eqref{E} and $F$ is defined in
\eqref{F}. Let $t^*$ be the maximal time of existence of
$u$. Then either
\begin{itemize}
    \item[(i)] $t^*=\infty$, or
    \item[(ii)] $t^*<\infty$ and $\lim_{t\to t^*}\|u(t)\|_{\mathbb{H}^2}=\infty$.
\end{itemize}
Moreover, for any $s<t^*$ the data-solution mapping $u_0\mapsto u$ is
continuous from $\mathbb{H}^2$ to $C([0,s];\mathbb{H}^2)$.
\end{proposition}
\begin{proof}
See Proposition 4.1 in \cite{pa}.
\end{proof}

Regarding global solutions in the energy space $\mathbb{H}^2$, we have the following result.

\begin{corollary}\label{globalteo}
Let $u_0\in \mathbb{H}^2$ be fixed. If $0<\alpha<8$, then the solution $u$ obtained in Theorem Proposition  \ref{localteo} can
be extended to the whole real line.
\end{corollary}
\begin{proof}
We need to use the conservation law $E(u(t))=E(u_0)$ for all $t\in [0,t^{\ast})$, and the Gagliardo-Nirenberg inequality in \cite[Chapter 3]{linares-ponce}. In fact, the term $||u(t)||_{\mathbb{L}^{\alpha+2}}^{\alpha+2}$ present in $(\ref{E})$ can be estimated by
\begin{equation}\label{GN1}
||u(t)||_{\mathbb{L}^{\alpha+2}}\leq 
C||u_{xx}(t)||_{\mathbb{L}^2}^{\theta}||u(t)||_{\mathbb{L}^2}^{1-\theta},
\end{equation}
where $\theta=\frac{\alpha}{4(\alpha+2)}$ and $C>0$ is a constant. By using $E(u(t))=E(u_0)$ and $F(u(t))=F(u_0)$, we obtain from inequality in $(\ref{GN1})$ that
\begin{equation}\label{GN2}
||u_{xx}(t)||_{\mathbb{L}^2}^2\leq\displaystyle E(u(t))+\int_{\mathbb{R}}|u(x,t)|^{\alpha+2}dx\leq E(u_0)+C||u_{xx}(t)||_{\mathbb{L}^2}^{\frac{\alpha}{4}}||u_0||_{\mathbb{L}^2}^{\frac{3\alpha+8}{4}}.
\end{equation}
\indent On the other hand, since $\left(||u_{xx}||_{\mathbb{L}^2}^2+||u||_{\mathbb{L}^2}^2\right)^{\frac{1}{2}}$ is an equivalent 
norm in $\mathbb{H}^2$, we obtain by $(\ref{GN2})$ that the local solution $u$ obtained in Proposition $\ref{localteo}$ is global in the energy space $\mathbb{H}^2$ when $\alpha\in (0,8)$.
\end{proof}

\subsection{Orbital stability and instability}

We first make clear our notion of orbital stability.  Taking into account that \eqref{NLS-equation} is invariant by rotations $T_1(\theta)$ and translations $T_2(r)$, we define the orbit generated by $\Phi=(\varphi,0)$ as
\begin{equation}\label{l1}
\begin{split}
\Omega_\Phi&=\{T_1(\theta) T_2(r)\Phi;\;\;\theta,r\in\R\}\\
&= \left\{ \left(\begin{array}{cc}
 \cos\theta & \sin\theta\\
-\sin\theta & \cos\theta
\end{array}\right)\left(\begin{array}{c}
 \varphi(\cdot-r)\\
 0
\end{array}\right);\;\;\theta,r\in\R  \right\}.
\end{split}
\end{equation}
In $\Hh^2$, we introduce the pseudo-metric $d$ by
$
d(f,g):=\inf\{\|f-T_1(\theta)
T_2(r)\Phi\|_{\Hh^2};\;\theta,r\in\R\}.$

\indent We have to notice that the distance between $f$ and $g$ is the distance between $f$ and the orbit generated by  $g$ under the actions of rotation and translation, that is,
$d(f,\Phi)=d(f,\Omega_{\Phi}).$
 We have the following definition of orbital stability:

\begin{definition}\label{stadef}
Let $\Theta(x,t)=e^{i\omega t}\varphi(x)=(\varphi(x)\cos(\omega t), \varphi(x)\sin(\omega t))$ be a standing wave for \eqref{NLS-equation}. We say that $\Theta$ is orbitally stable in $\Hh^2$ provided that, given $\ve>0$, there exists $\delta>0$ with the following property: if $u_0\in \Hh^2$ satisfies $\|u-\Phi\|_{\Hh^2}<\delta$, then the solution, $u(t)$, of \eqref{NLS-equation} with initial condition $u_0$ exist for all $t\geq0$ and satisfies
$
d(u(t),\Omega_\Phi)<\ve,\ \mbox{for all}\ t\geq0.
$
Otherwise, we say that $\Theta$ is orbitally unstable in $\Hh^2$.
\end{definition}

\indent We have the following stability result.

\begin{proposition}\label{teostab1} Let $\varphi$ be an even minimizer of the problem $\ref{min-problem}$. If ${\rm z}(\mathcal{L})=2$, $d''(\omega)>0$ and the Cauchy problem in $(\ref{CauchyNLS})$ is globally well-posed in the energy space $\mathbb{H}^2$, then the minimizer $\varphi$ is orbitally stable in the sense of Definition $\ref{stadef}$.
\end{proposition}
\begin{proof}
Since $\varphi$ is an even minimizer of the problem $(\ref{min-problem})$, we obtain by the proof of Propositon $\ref{specL1}$ that ${\rm n}(\mathcal{L})=1$ and ${\rm z}(\mathcal{L})=2$ with $\Ker(\mathcal{L})=[(\varphi',0),(0,\varphi)]$. Since $d''(\omega)>0$ and the Cauchy problem in $(\ref{CauchyNLS})$ is globally well-posed in the energy space $\mathbb{H}^2$, we obtain by the arguments in \cite{we} (see also \cite{FabioAdemir}) that $\varphi$ is orbitally stable in the sense of Definition $\ref{stadef}$.
\end{proof}

\begin{remark}\label{operaeven} Let $\mathbb{H}_e^2$ be the space constituted by complex functions $v\in\mathbb{H}^2$ such that $v$ is even. In $\mathbb{H}_e^2$, we can consider the Cauchy problem in $(\ref{CauchyNLS})$ to study the local and global well-posedness of even solutions in the sense that the results of Proposition $\ref{localteo}$ and Corollary $\ref{globalteo}$ can be easily adapted for this particular case. The only feature that we lost in our analysis is that over $\mathbb{H}_e^2$, the equation in $(\ref{CauchyNLS})$ is not invariant under the symmetry of translations, and this means that it is necessary to remove  the action $T_2$ in the Definition of stability $\ref{stadef}$. Since the associated operator $\mathcal{L}_e=\left(\begin{array}{cccc}\mathcal{L}_{-,e}&  0\\
0& \mathcal{L}_{+,e}\end{array}\right): D(\mathcal{L}_e)=\mathbb{H}_e^2\subset \mathbb{L}_e^2\rightarrow \mathbb{L}_e^2$ has the same form as $\mathcal{L}$ in $(\ref{operator})$ and is invariant over the space $\mathbb{L}_e^2$ constituted by functions in $\mathbb{L}^2$ that are even, we obtain immediately when $\varphi$ is an even minimizer of the problem $(\ref{min-problem})$ that ${\rm n}(\mathcal{L}_e)=1$ (this would happen when $\alpha$ is an even integer according to Remark $\ref{remarkeven}$). On the other hand, if $\varphi$ is even, we obtain that $\varphi'$ is odd and this implies $(\varphi',0)$ is not an element of $\Ker(\mathcal{L}_e)$. \end{remark}

\begin{remark}
\label{stabmin} Let $\varphi$ be an even minimizer of the problem $\ref{min-problem}$ and suppose ${\rm z}(\mathcal{L}_{e})=1$. If  $d''(\omega)>0$, then the minimizer $\varphi$ is orbitally stable in $\mathbb{H}_e^2$ in the sense of Definition $\ref{stadef}$ only modulus rotations. This fact follows by \cite[Theorem 3.5]{gss1}. Now, if  $d''(\omega)<0$, then the minimizer $\varphi$ is orbitally unstable in $\mathbb{H}_e^2$ in the sense of Definition $\ref{stadef}$ only modulus rotations. In particular, $\phi$ is orbitally unstable in $\mathbb{H}^2$ in the sense of Definition $\ref{stadef}$. This fact is a consequence of \cite[Theorem 4.7]{gss1}.

\end{remark}

\subsection{Numerical study of $d''(\omega)$} Now, we will study the sign of 
$\displaystyle d''(\omega)=\frac{1}{2}\frac{d}{d\omega}\int_ {\mathbb{R}} \varphi^2(x) dx,$
numerically. To give an analytical support for this study in a neighbourhood of $\omega_0$, we first need to prove the following result:
\begin{proposition}\label{propsmooth}
 Let $\varphi$ be the solution of $(\ref{edo})$ given by $(\ref{expsolode})$. For each $\alpha>0$, there exists a neighborhood $I_{\alpha}$ centered at $\omega_0$, an open ball $B$ in $H^4$ centered at $\varphi$, and a unique smooth curve $\omega\in I_{\alpha}\mapsto \varphi_{\omega}\in B\subset H^4$ of even positive solutions for the equation $(\ref{edo})$. 
\end{proposition}
\begin{proof}
Let us define the map $\Upsilon:(0,+\infty)\times H_{e}^4\rightarrow L_{e}^2$ given by $\Upsilon(\omega,\psi)=\psi''''-\psi''+\omega\psi-\psi^{\alpha+1},
$
where $L_{e}^4$ and $H_{e}^4$ are the spaces $L^2$ and $H^4$, respectively, constituted by even functions. We see from equation $(\ref{edo})$ that $\Upsilon(\omega_0,\varphi)=0$. Additionally, the function $\Upsilon$ is smooth in all variables, and the Fr\'echet derivative at the point $(\omega_0,\varphi)$ is 
$D_{\psi}\Upsilon(\omega_0,\varphi)=\partial_x^4-\partial_x^2+\omega_0-(\alpha+1)\varphi^{\alpha}=\mathcal{L}_{-,e}.$
Since $\varphi$ is even, $\varphi'$ is odd and it is not an element in $H_{e}^4$. By Proposition $\ref{specL1}$, we obtain that $0$ is an isolated simple eigenvalue associated with the eigenfunction 
$\varphi'$, the remainder of the spectrum is bounded away from zero, and the essential spectrum is the interval $[\omega_0,+\infty)$. Thus, we conclude that $0$ is not an element of the spectrum of the operator $\mathcal{L}_{-,e}$ over the space $L_{e}^2$, which implies that $0$ belongs to the resolvent set of $\mathcal{L}_{-,e}$. From the implicit function theorem, we obtain a neighborhood $I_{\alpha}$ centered at $\omega_0$, an open ball $B$ in $H^4$ centered at $\varphi$, and a unique smooth curve $\omega\in I_{\alpha}\mapsto \varphi_{\omega}\in B\subset H^4$ of solutions for the equation $(\ref{edo})$. This fact proves the proposition.
\end{proof}

\begin{remark}
By Proposition $\ref{propsmooth}$, if $\omega$ is sufficiently close to $\omega_0$, we still obtain that ${\rm z}(\mathcal{L}_{-}) = 1$. For a minimizer $\varphi_{\omega}$ which depends on $\omega>0$ of the problem $(\ref{minimizationproblem})$, where $\omega$ is far from $\omega_0$, we can still guarantee that ${\rm n}(\mathcal{L}_{-}) = 1$ and ${\rm n}(\mathcal{L}_{+}) = 0$ (see Proposition $\ref{propunique}$). However, we cannot conclude that ${\rm z}(\mathcal{L}_{-}) = 1$ in order to establish the existence of a smooth curve $\omega\in(0,+\infty)\mapsto\varphi_{\omega}$ depending on $\omega$, as obtained in Proposition $\ref{propsmooth}$. The existence of a smooth curve as above is crucial for determining the sign of $d''(\omega)$. This property is also a cornerstone in the stability analysis, as discussed in Proposition $\ref{teostab1}$ and Remark $\ref{stabmin}$.
\end{remark}
\indent 
To explore the sign of  $d''(\omega)$, the integral is approximated by the trapezoidal rule and the derivative is calculated by a forward difference approximation. Figure \ref{d-second-der} shows the values of 
$\displaystyle d''(\omega)$ for $\alpha=2,~3,~4, ~5$.
The dashed vertical lines in the rest of the figures correspond to the $\omega_0$ value where the exact solution exists. The graphs show that $d''(\omega)>0$ when $\alpha \leq 4$; however, the sign changes from negative to positive at a threshold $\omega$ value for $\alpha=5.$  
Therefore, we next focus on the interval $\alpha \in (4,5)$.  Figure \ref{d-second-der2} illustrates the variation of $d''(\omega)$ with $\omega$ for $\alpha=4.1$.
The experiment shows that the sign-changing behavior begins with $\alpha\approx 4.1$. As it can be seen from Figure \ref{alpha_d} the sign of $d''(\omega)$ changes at $\omega=\omega_0$  for
 $\alpha \approx 4.8$.  Finally, we consider the values where $\alpha >5$. We present the  variation of $d''(\omega)$ with $\omega$ for $\alpha=5.4$ and $\alpha=5.5$ in Figure \ref{d-second-der3}. The result shows that $d''(\omega)$ becomes negative for all $\omega$ values when $\alpha \geq 5.5$.  Our results indicate a prediction for the orbital stability since for $\alpha \in (0, 4.1)$, we have $d''(\omega)>0$. If $\alpha \in [4.1, 5.5)$, we obtain a threshold value $\omega_c$ such that if $\omega\in (0,\omega_c)$, we have $d''(\omega)<0$, and if $\omega>\omega_c$, we obtain $d''(\omega)>0$. At $\omega=\omega_c$, we deduce that $d''(\omega_c)=0$, that is, we obtain a frequency $\omega$ that indicates a critical value for the stability. When $\alpha \geq 5.5 $, we obtain that $d''(\omega)<0$. 
\begin{figure}[h!]
 \begin{minipage}[t]{0.45\linewidth}
   \includegraphics[width=3.1in]{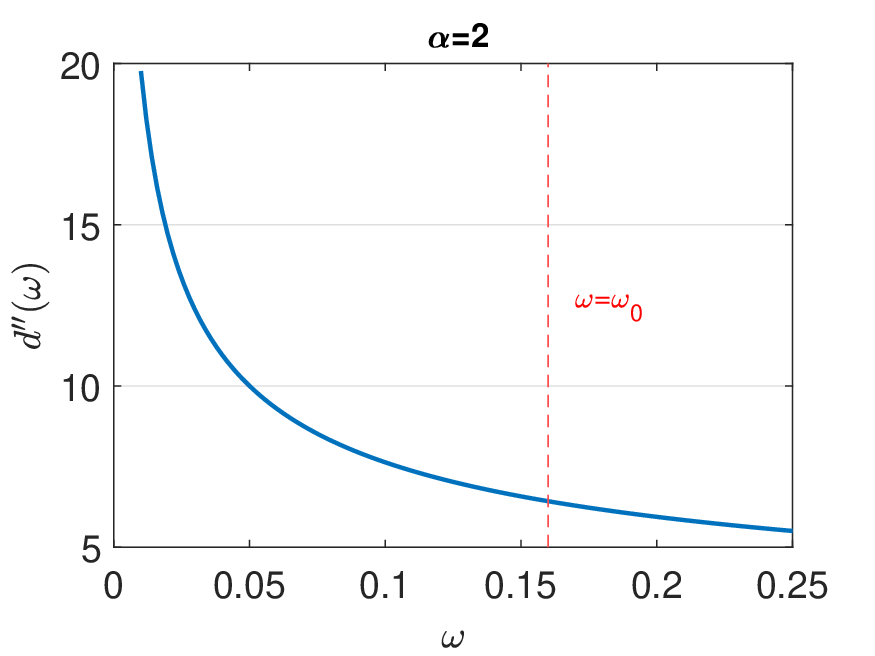}
 \end{minipage}
\hspace{30pt}
\begin{minipage}[t]{0.45\linewidth}
   \includegraphics[width=3.1in]{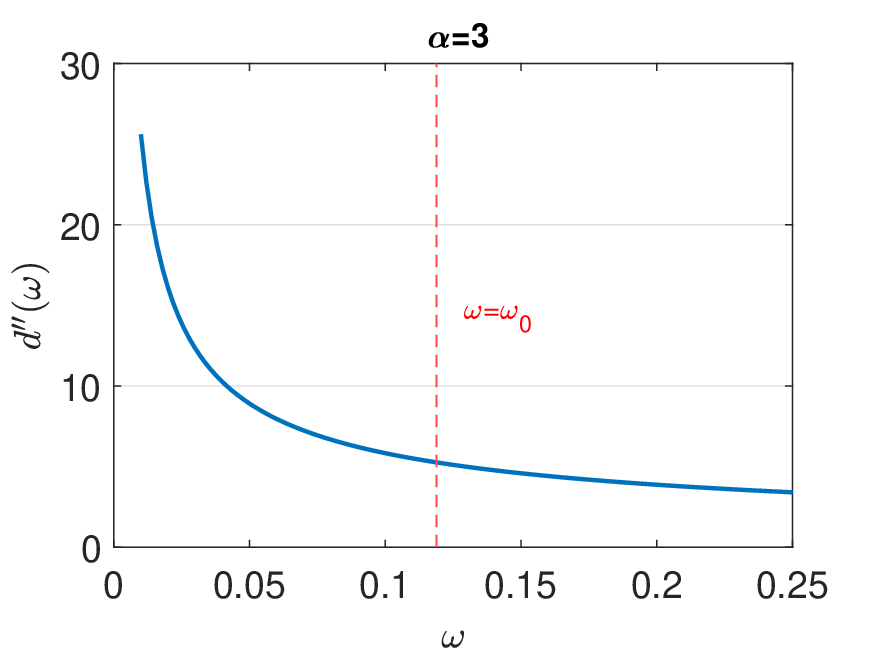}
 \end{minipage}
 \hspace{30pt}
\begin{minipage}[t]{0.45\linewidth}
   \includegraphics[width=3.1in]{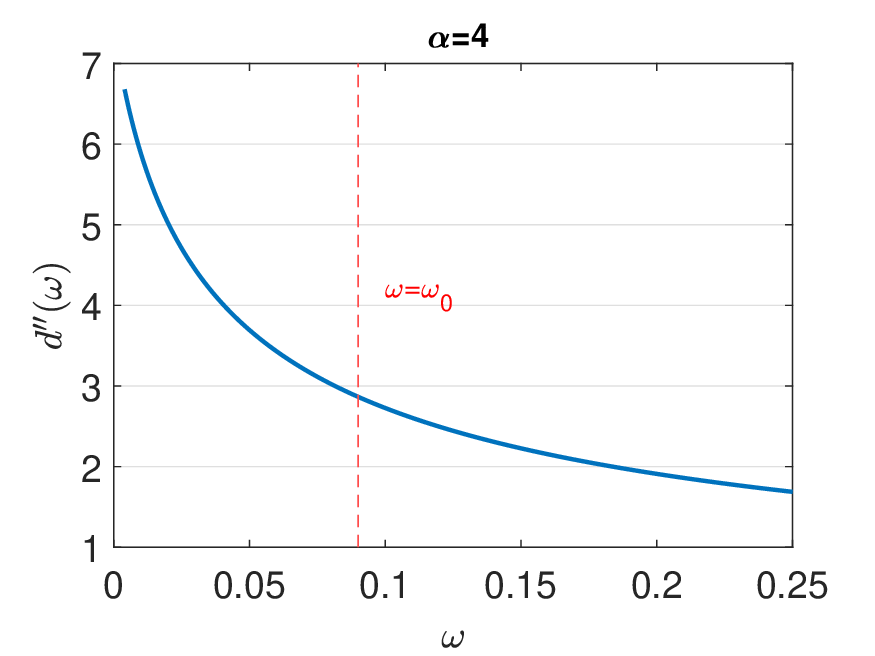}
 \end{minipage}
 \hspace{30pt}
\begin{minipage}[t]{0.45\linewidth}
   \includegraphics[width=3.1in]{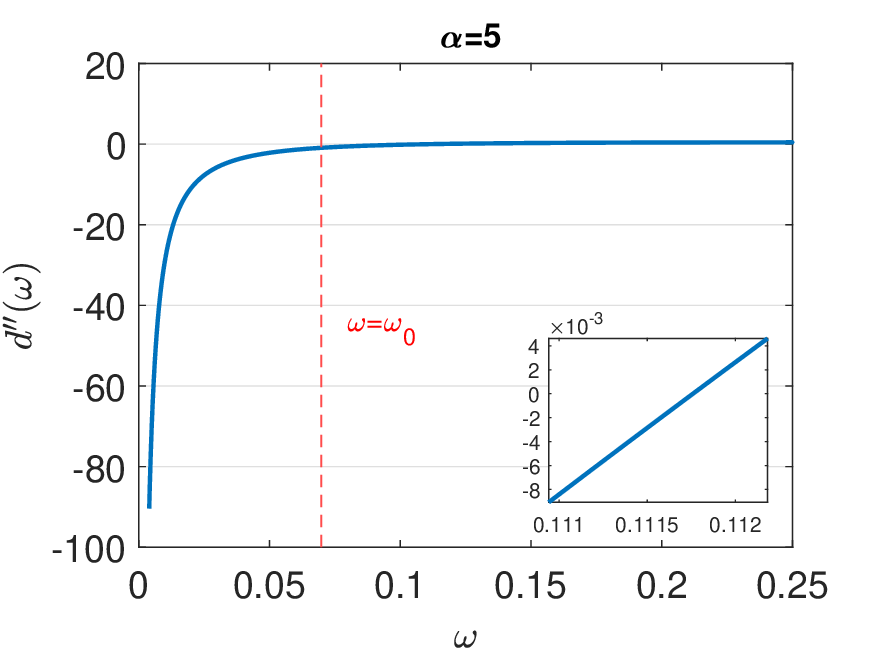}
 \end{minipage}
 \caption{ The variation of $d''(\omega)$ with $\omega$ for $\alpha=2,~3,~4, ~5$.}
 \label{d-second-der}
\end{figure}

\begin{figure}[h!]
  \includegraphics[width=3.1in]{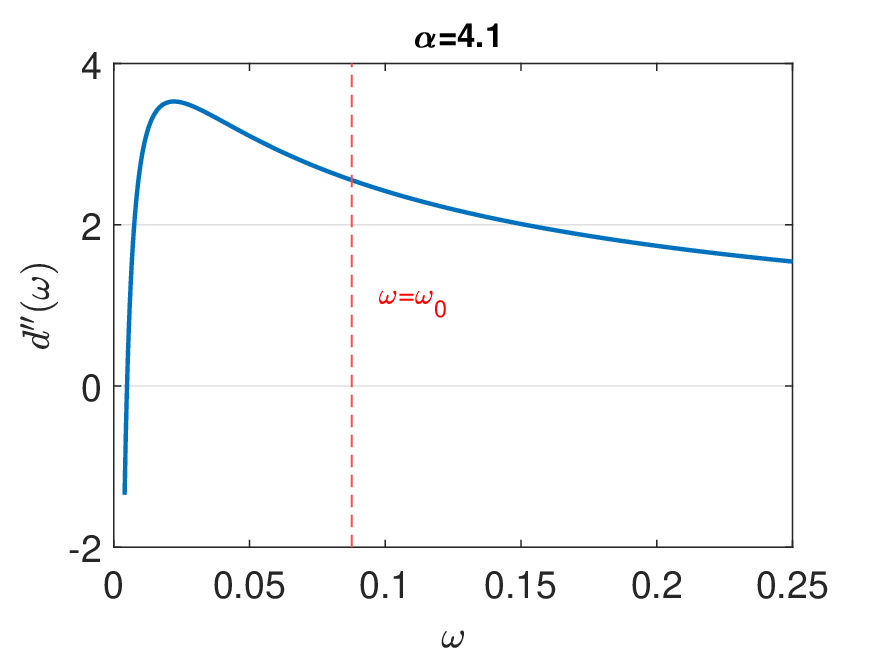}
  \caption{The variation of $d''(\omega)$ with $\omega$ for $\alpha=4.1$.}
 \label{d-second-der2}
\end{figure}

\begin{figure}[h!]
 \begin{minipage}[t]{0.45\linewidth}
  \includegraphics[width=3.1in]{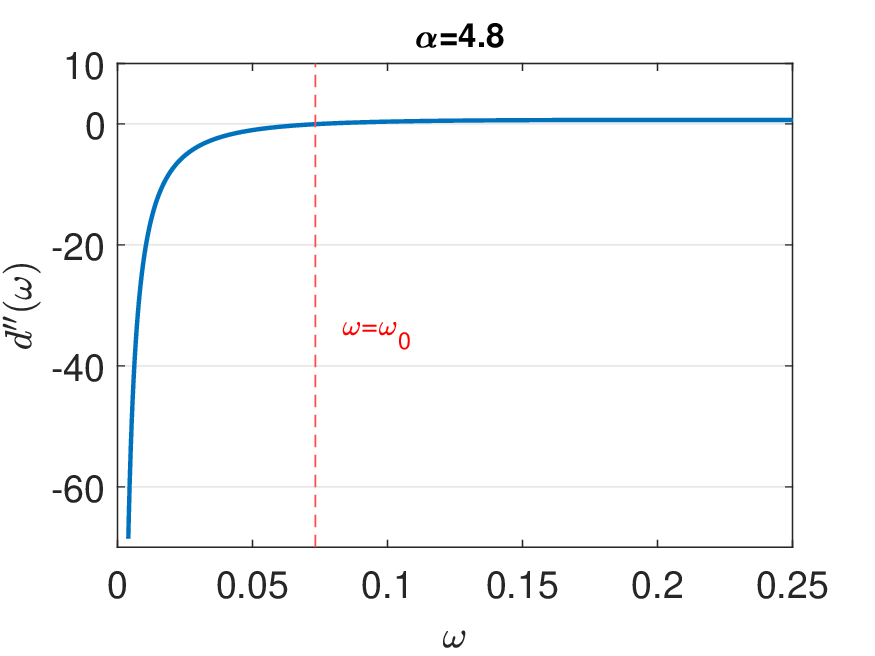}
 \end{minipage}
\hspace{30pt}
\begin{minipage}[t]{0.45\linewidth}
   \includegraphics[width=3.1in]{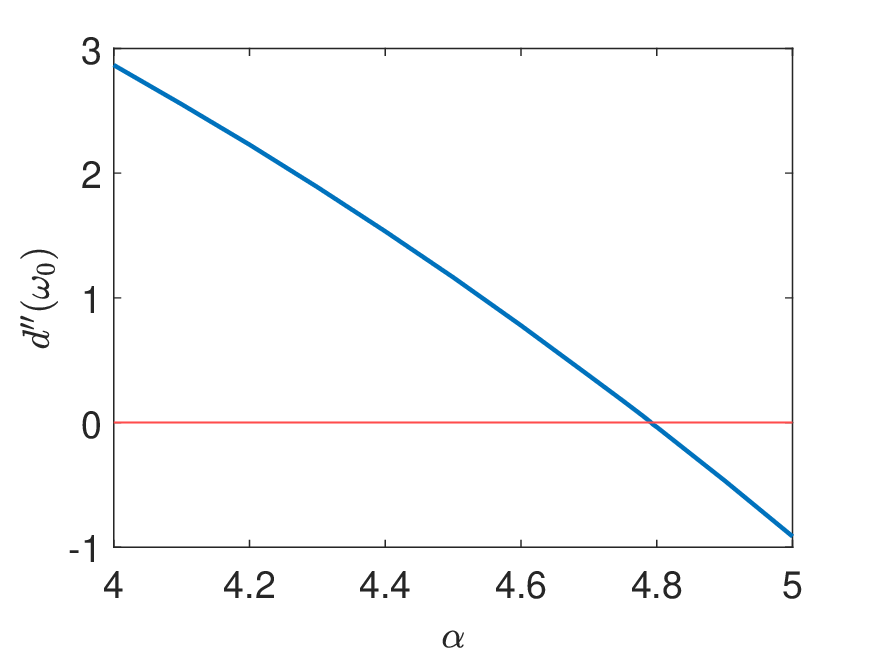}
 \end{minipage}
  \caption{The variation of $d''(\omega)$ with $\omega$ for $\alpha=4.8$ and the variation of  $d''(\omega_0)$ with $\alpha$.}
 \label{alpha_d}
\end{figure}

\begin{figure}[h!]
 \begin{minipage}[t]{0.45\linewidth}
  \includegraphics[width=3.1in]{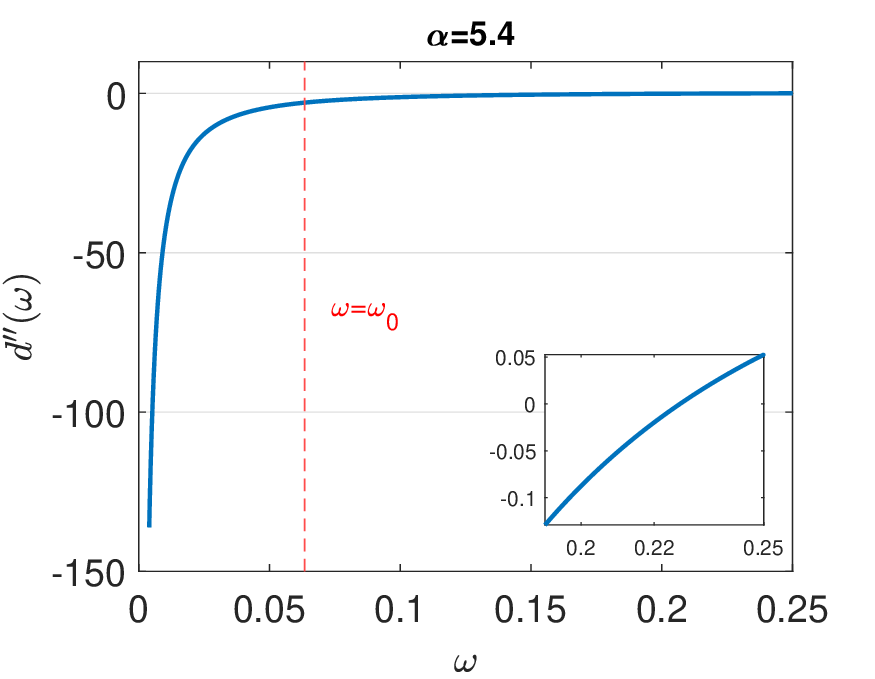}
 \end{minipage}
\hspace{30pt}
\begin{minipage}[t]{0.45\linewidth}
   \includegraphics[width=3.1in]{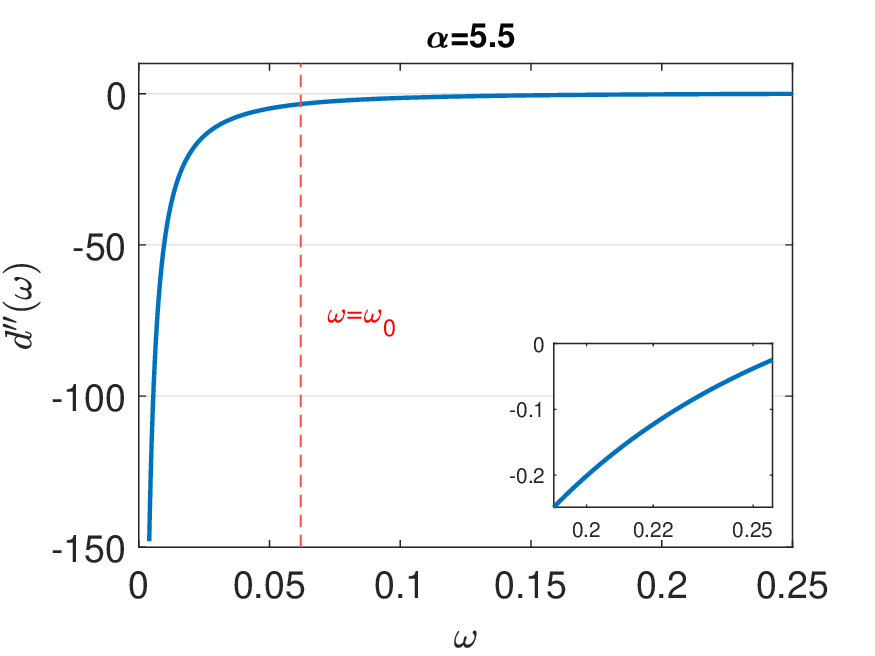}
 \end{minipage}
  \caption{The variation of $d''(\omega)$ with $\omega$ for $\alpha=5.4$ and $\alpha=5.5$.}
 \label{d-second-der3}
\end{figure}

\begin{figure}[h!]
   \includegraphics[width=3.5in]{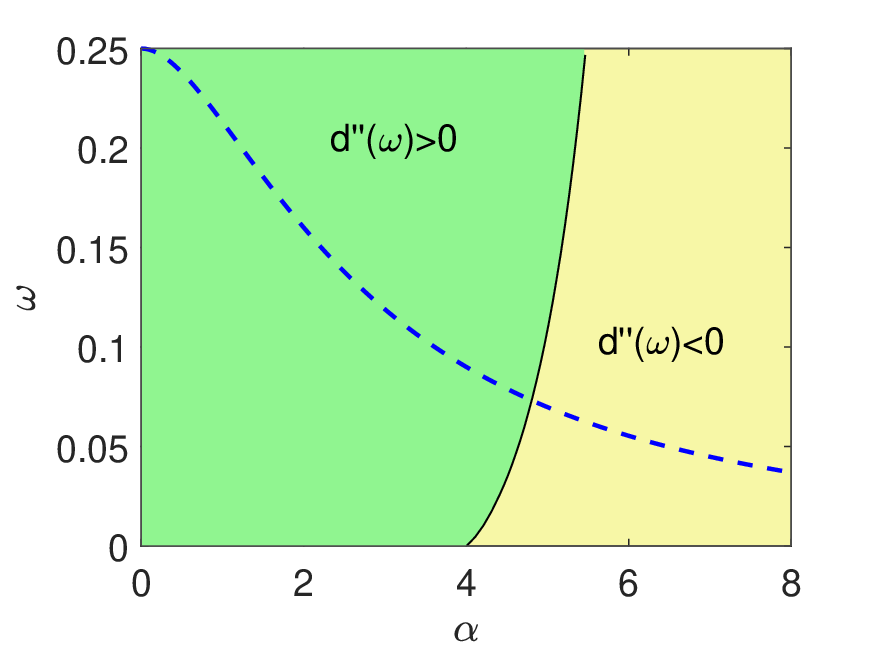}
  \caption{The regions where $d''(\omega)>0$ (green) and $d''(\omega)<0$ (yellow) in the $(\alpha, \omega)$ plane. The dashed blue curve represents the graph of $d''(\omega)$ associated with the solitary wave in 
  $(\ref{expsolode})$. The intersection with the curve predicting stability (green) and instability (yellow) occurs at $(\omega_0, \alpha_0) \approx (0.07, 4.8)$.}

 \label{regions}
\end{figure}

In order to summarize the above results, Figure \ref{regions} shows the sign of $d''(\omega)$ in the $(\alpha, \omega)$ plane. In this figure, the dashed blue curve corresponds to $(\alpha, \omega_0)$ values in \eqref{omega} where there exist explicit solutions. This curve intersects the boundary of the regions where $d''(\omega)>0$ and $d''(\omega)<0$ at $\alpha=4.8$.


\textit{Proof of Claims $\rm{(C1)}$-$\rm{(C2)}$ on page 4.} By Proposition $\ref{propmain}$, we obtain that ${\rm n}(\mathcal{L})=1$ and ${\rm z}(\mathcal{L})=2$. In addition, since $\phi'=\varphi'$ is odd, we obtain ${\rm n}(\mathcal{L}_e)=1$ and ${\rm z}(\mathcal{L}_e)=1$. According to Figures \ref{d-second-der}-\ref{regions},  we see that $d''(\omega_0)>0$ for all $\alpha\in(0,\alpha_0)$, while $d''(\omega_0)<0$ for all $\alpha>\alpha_0$. Important to mention that it is possible to see in Figure $\ref{alpha_d}$ and in Figure \ref{regions} there is a threshold value $\alpha_0$ where $d''(\omega_0)=0$. For $\alpha<\alpha_0$, we can use Proposition $\ref{teostab1}$ to conclude that $\varphi$ in $(\ref{expsolode})$ is orbitally stable in $\mathbb{H}^2$. On the other hand, if $\alpha>\alpha_0$ we obtain that $\varphi$ is orbitally unstable in $\mathbb{H}^2$ according to the Remark $\ref{stabmin}$.
\begin{flushright}
$\blacksquare$
\end{flushright}
\newpage

 \begin{remark} Some important considerations need to be highlighted. The first one is that the presence of the mixed term $\partial_x^4-\partial_x^2$ in equation $(\ref{NLS-equation})$ causes the surprising fact that even though the parameter $\alpha$ belongs to an interval where the existence of global solutions for the Cauchy problem $(\ref{CauchyNLS})$, namely $\alpha \in (\alpha_0,8)$, we have that the minimizer $\varphi$ in $(\ref{expsolode})$ is orbitally unstable in $\mathbb{H}^2$. This fact is a novelty compared with the  NLS equation $(\ref{2NLS-equation})$ since parameter $\alpha$ for the existence of a global solution in $\mathbb{H}^1$ and the orbital stability are the same, that is, if $\alpha\in (0,4)$ we have global solutions and stability and if $\alpha\in (4,+\infty)$ we obtain blow-up phenomena and instability. As far as we know, we have a similar scenario if we remove the term $u_{xx}$ in $(\ref{NLS-equation})$ to deal with the equation
 \begin{equation}\label{4NLS}
iu_t-u_{xxxx}+|u|^{\alpha}u=0.
 \end{equation}
 The scaling invariance $u\mapsto \lambda^{\frac{4}{\alpha}}u(\lambda x,\lambda^4t)$ associated with the equation $(\ref{4NLS})$ predicts that we would have global solutions and orbital stability for $\alpha\in (0,8)$. Blow-up phenomena may occur when $\alpha> 8$. On the other hand, the prediction for the stability/instability can be easily determined since we have the equality
 \begin{equation}\label{d24NLS}
d''(\omega)=\frac{1}{2\omega}\left(\frac{8-\alpha}{2\alpha}\right)\int_{\mathbb{R}}\varphi^2(x)dx.
 \end{equation}

\end{remark}

\vskip 1cm

\section*{Acknowledgments}
F. Natali is partially supported by CNPq/Brazil (grant 303907/2021-5).

\end{document}